\documentclass[12pt]{article}
\usepackage[margin =1in]{geometry}
\usepackage{amssymb,amsthm}
\usepackage{amsmath}
\usepackage{enumerate}
\usepackage{hyperref}
\usepackage{cite}
\usepackage{color}

\numberwithin{equation}{section}

\newcommand{\inclu}[0] {\ar@{^{(}->}}

\newcommand{\gph}{{\rm gph}\,}

\newcommand{\dist}{{\rm dist}}
\newcommand{\R}{\mathbb{R}}

\newcommand{\cR}{\mathcal{R}}
\newcommand{\EE}{\mathbb{E}}

\newcommand{\sign}{\mathrm{sign}}

\newcommand{\RR}{\mathbb{R}}
\newcommand{\pfail}{p_{\mathrm{fail}}}
\newcommand{\cX}{\mathcal{X}}

\newcommand{\cN}{\mathcal{N}}

\newcommand{\lipsymb}{\mathsf{L}}
\newcommand{\probspace}{\mathcal{Z}}
\newcommand{\domainspace}{\R^d}
\newcommand{\alg}{\mathcal{A}}
\newcommand{\bregmetric}{d_{\Phi}}
\newcommand{\bregsym}{D_{\Phi}^\mathrm{sym}}
\newcommand{\inter}{{\rm int}}

\newcommand{\cl}{\textrm{cl}\,}
\newcommand{\ri}{{\rm ri}\,}
\newcommand{\la}{\langle}
\newcommand{\ra}{\rangle}

\newcommand{\cG}{\mathcal{G}}
\newcommand{\cS}{\mathcal{S}}

\newcommand{\prox}{\mathrm{prox}}

\newcommand{\dom}{\mathrm{dom}\,}

\newcommand{\argmin}{\operatornamewithlimits{argmin}}

\newtheorem{thm}{Theorem}[section]
\newtheorem{definition}[thm]{Definition}

\newtheorem{lem}[thm]{Lemma}
\newtheorem{cor}[thm]{Corollary}

\newtheorem{example}{Example}[section]

\theoremstyle{remark}
\newtheorem{claim}{Claim}

\usepackage{mathtools}
\DeclarePairedDelimiter{\dotp}{\langle}{\rangle}
\usepackage[boxruled]{algorithm2e}

\begin{document}
	
		\title{Graphical Convergence of Subgradients in Nonconvex Optimization and Learning}
	
	\author{Damek Davis
		\thanks{School of Operations Research and Information Engineering, Cornell University, Ithaca, NY 14850; 	
		\texttt{people.orie.cornell.edu/dsd95/}.} 
	\and
	Dmitriy Drusvyatskiy
		\thanks{Department of Mathematics, University of Washington, Seattle, WA 98195; 	
		\texttt{sites.math.washington.edu/{\raise.17ex\hbox{$\scriptstyle\sim$}}ddrusv}. Research of Drusvyatskiy was partially supported by the AFOSR YIP award FA9550-15-1-0237 and by the NSF DMS   1651851 and CCF 1740551 awards.}}

	\date{}
	\maketitle

\begin{abstract}
We investigate the stochastic optimization problem of minimizing population risk, where the loss defining the risk is assumed to be weakly convex. Compositions of Lipschitz convex functions with smooth maps are the primary examples of such losses. We analyze the estimation quality of such nonsmooth and nonconvex problems by their sample average approximations. Our main results establish dimension-dependent rates on subgradient estimation in full generality and dimension-independent rates when the loss is a generalized linear model. As an application of the developed techniques, we analyze the nonsmooth landscape of a robust nonlinear regression problem.  
\end{abstract}


\section{Introduction.}
Traditional machine learning theory quantifies how well a decision rule, learned from a limited data sample, generalizes to the entire population. The decision rule itself may enable the learner to correctly classify (as in image recognition) or predict the value of continuous statistics (as in regression) of previously unseen data samples.
A standard mathematical formulation of this problem associates to each decision rule $x $ and each sample $z$, a loss $f(x, z)$, which may for example penalize misclassification of the data point by the decision rule.  Then the learner seeks to minimize the \emph{regularized population risk}: 
\begin{align}\label{eq:population}
\min_{x} \; \varphi(x)=f(x) +r(x)\qquad \textrm{ where }\qquad  f(x)=\EE_{z \sim P}\left[ f(x, z) \right].
\end{align}
Here,   $r\colon\R^d\to\R\cup\{+\infty\}$ is an auxiliary function defined on $\R^d$ that may encode geometric constraints  or promote low-complexity structure (e.g., sparsity or low-rank) on $x$. 
The main assumption is that the only access to the population data  is by drawing i.i.d. samples from $P$. Numerical methods then seek to obtain a high-quality solution estimate for \eqref{eq:population} using as few samples as possible. 
Algorithmic strategies for \eqref{eq:population} break down along two lines: streaming strategies and regularized empirical risk minimization (ERM). 

Streaming algorithms in each iteration update a solution estimate of \eqref{eq:population} based on drawing a relatively small batch of samples. Streaming algorithms deviate from each other in
precisely how the sample is used in the update step. The proximal stochastic subgradient method \cite{ghadimilanzang,robust_stoch_opt,tame_paper} is one popular streaming algorithm, although there are many others, such as the stochastic proximal point and Gauss-Newton methods \cite{stochastic_subgrad,duchi_ruan,toulis2}. 
In contrast,  ERM-based algorithms  draw a large sample $S=\{z_1,z_2,\ldots,z_m\}$ at the onset and  output the solution of the deterministic problem 
\begin{align}\label{eq:empirical0}
\min_{x \in \domainspace} \; \varphi_S(x):= f_S(x)+r(x)\qquad \textrm{ where }\qquad f_S(x) := \frac{1}{m} \sum_{i=1}^m f(x, z_i).
\end{align}
Solution methodologies for \eqref{eq:empirical0} depend on the structure of the loss function. One generic approach, often used in practice, is to apply a streaming algorithm directly to \eqref{eq:empirical0}  by interpreting 
$f_S(\cdot)$ as an expectation over the discrete distribution on the samples $\{z_i\}_{i=1}^n$ and
performing multiple passes through the sampled data. Our current work focuses on the ERM strategy, though it is strongly influenced by recent progress on streaming algorithms.

The success of the ERM approach rests on knowing that the minimizer of the surrogate problem \eqref{eq:empirical0} is nearly optimal for the true learning task \eqref{eq:population}.
Quantitative estimates of this type are often based on a uniform convergence principle. 
For example, when the functions $f(\cdot,z)$ are  $L$-Lipschitz continuous for a.e. $z\sim P$, then with probability $1-\gamma$, the estimate holds \cite[Theorem 5]{stocH-shal}:
\begin{equation}\label{eqn:est_gener}
\sup_{x:\, \|x\|_2\leq R} ~|f(x)-f_S(x)|= {\mathcal{O}}\left(\sqrt{\frac{L^2R^2 d\log(m)}{m}\cdot \log \left(\frac{d}{\gamma}\right)}\right).
\end{equation}

An important use of the bound in \eqref{eqn:est_gener} is to provide a threshold beyond which algorithms for the surrogate problem \eqref{eq:empirical0} should terminate, since further 
accuracy on the ERM may fail to improve the accuracy on the true learning task. It is natural to ask if under stronger assumptions, learning is possible with sample complexity that is independent of the ambient dimension $d$. In the landmark paper \cite{stocH-shal}, the authors showed that the answer is indeed yes when the functions $f(\cdot,z)$ are convex and one incorporates further strongly convex regularization. Namely, under an appropriate choice of the parameter $\lambda>0$, the solution of the quadratically regularized problem 
\begin{equation}\label{eqn:quadrag}
\hat x_S:=\argmin_{x\in \R^d} \left\{\varphi_S(x)+\lambda\|x\|^2_2\right\}, 
\end{equation}
satisfies 
\begin{equation}\label{eqn:dim_indep_bound}
\varphi(\hat x_S)-\inf  \varphi\leq \sqrt{\frac{8L^2R^2}{\gamma m}}
\end{equation} 
with probability $1-\gamma$, where $R$ is the diameter of the domain of $r$. In contrast to previous work, the proof of this estimate is not based on uniform convergence. Indeed, uniform convergence in function values may fail in infinite dimensions even for convex learning problems. Instead, the property underlying the dimension independent bound \eqref{eqn:dim_indep_bound} is that the solution $\hat x_S$ of the quadratically regularized ERM \eqref{eqn:quadrag} is stable in the sense of Bousquet-Elisseff \cite{stability_baquett}. That is, the solution $\hat x_S$  does not vary much under an arbitrary perturbation of a single sample $z_i$. Stability of quadratically regularized ERM will also play a central role in our work for reasons that will become clear shortly.

The aforementioned bounds on the accuracy of regularized ERM are only meaningful if one can globally solve the deterministic problems \eqref{eq:empirical0} or \eqref{eqn:quadrag}. Convexity certainly facilitates global optimization techniques.
Many problems of contemporary interest, however, are nonconvex, thereby making ERM-based learning rules intractable.
When the functions $f(\cdot,z)$ are not convex but smooth, the most one can hope for is to find points that are critical for the problem \eqref{eq:empirical0}. Consequently, it may be more informative to estimate the deviation in the gradients, $\sup_{x: \|x\|_2\leq R}\|\nabla f(x)-\nabla {f_S}(x)\|_2$, along with deviations in higher-order derivatives when they exist.  Indeed, then in the simplest setting $r=0$, the standard decomposition 
$$
\|\nabla f(x)\|_2 \leq \underbrace{\|\nabla f(x)-\nabla {f_S}(x)\|_2}_{\text{generalization error}} + \underbrace{\|\nabla {f_S}(x)\|_2}_{\text{optimization error}},
$$
relates near-stationarity for the empirical risk to near-stationarity for the population risk.
Such uniform bounds have recently appeared in \cite{landscape_grad_mont,karthik_new}.

When the loss $f(\cdot,z)$ is neither smooth nor convex, the situation becomes less clear. Indeed, one should reassess what ``uniform convergence of gradients'' should mean in light of obtaining termination criteria for algorithms on the regularized ERM problem. As the starting point, one may replace the gradient by a generalized subdifferential $\partial \varphi(x)$ in the sense of nonsmooth and  variational analysis \cite{RW98,Mord_1}. Then the minimal norms, $\dist(0,\partial \varphi(x))$ and $\dist(0,\partial \varphi_S(x))$, could serve as stationarity measures akin to the norm of the gradient in smooth minimization. One may then posit that the stationarity measures, $\dist(0,\partial \varphi(x))$ and $\dist(0,\partial \varphi_S(x))$, are uniformly close with high probability when the sample size is large. Pointwise convergence is indeed known to hold (e.g. \cite[Theorem 7.54]{stoch_prog}). On the other hand, to our best knowledge, all results on uniform convergence of the stationarity measure are asymptotic and require extra assumptions, such as polyhedrality for example \cite{ralph_xu}. The main obstacle is that the function $x\mapsto\dist(0,\partial \varphi(x))$ is highly discontinuous.
We refer the reader to \cite[pp. 380]{stoch_prog} for a discussion.
Indeed, the need to look beyond 
pointwise uniform convergence is well-documented in  optimization and variational analysis \cite{wets_book,attouch2}. One remedy is to instead focus on graphical convergence concepts. Namely, one could posit that the Hausdorff distance between the subdifferential graphs, $\gph \partial \varphi$ and $\gph \partial \varphi_S$, tends to zero.
Here, we take a closely related approach,  
while aiming for finite-sample bounds. 

In this work, we aim to provide tight threshold estimates beyond which  algorithms on \eqref{eq:empirical0} should terminate. In contrast to previous literature, however, we will allow the loss function to be both nonconvex and nonsmooth. The only serious assumption we make is that $f(\cdot,z)$ is a $\rho$-weakly convex function for a.e. $z\sim P$, by which we mean that the assignment $x \mapsto f(x, z) + \frac{\rho}{2}\|x\|^2_2$ is convex. The class of weakly convex functions is broad and its importance in optimization  is well documented \cite{fav_C2,prox_reg,Nurminskii1973,paraconvex,semiconcave}.\footnote{Weak convexity also goes by other names such as lower-$C^2$, uniformly prox-regular, and semiconvex.} It trivially includes all convex functions and all $C^1$-smooth functions with Lipschitz gradient. More broadly, it includes all  compositions $f(x,z)=h(c(x,z),z)$, where $ h(\cdot,z) $ is convex and Lipschitz, and $c(\cdot,z)$ is $C^1$-smooth with Lipschitz Jacobian. Robust principal component analysis,  phase retrieval,  blind deconvolution, sparse dictionary learning, and minimization of risk measures naturally lead to stochastic weakly convex minimization problems. We refer the interested reader to \cite[Section 2.1]{stochastic_subgrad} and \cite{prox_point_surv} for detailed examples. 

The approach we take is based on a smoothing technique, familiar to optimization specialists.
For any function $g$, define the Moreau envelope and the proximal map: 
$$g_{\lambda}(x):=\min_{y} \left\{g(y)+\frac{1}{2\lambda}\|y-x\|^2_2\right\},\qquad \prox_{\lambda g}(x):=\argmin_{y} \left\{g(y)+\frac{1}{2\lambda}\|y-x\|^2_2\right\}.$$
It is well-known that if $g$ is $\rho$-weakly convex and $\lambda<\frac{1}{\rho}$, then the envelope $g_\lambda$ is $C^1$-smooth with gradient
$$\nabla g_\lambda(x)=\lambda^{-1}(x-\prox_{\lambda g}(x)).$$
Note that $\nabla g_\lambda(x)$ is in principle computable by solving a  convex optimization problem in the definition of the proximal point $\prox_{\lambda g}(x)$. 

Our main result (Theorem~\ref{thm:prox_close_p}) shows that with probability $1-\gamma$, the estimate holds: 
\begin{equation}\label{eqn:main_result}
\sup_{x:\, \|x\|_2\leq R} ~\|\nabla \varphi_{1/2\rho}(x)-\nabla (\varphi_S)_{1/2\rho}(x)\|_2= {\mathcal{O}}\left(\sqrt{\frac{L^2 d}{m}\log\left(\frac{R\rho m}{\gamma}\right)}\right).
\end{equation}
The guarantee \eqref{eqn:main_result} is appealing: even though the subgradients of $\varphi$ and $\varphi_S$ may be far apart pointwise, the gradients of the smooth approximations $ \varphi_{1/2\rho}$ and $(\varphi_S)_{1/2\rho}$ are uniformly close at a controlled rate governed by the sample size. Moreover, \eqref{eqn:main_result} directly implies estimates on the Hausdorff distance between subdifferential graphs, $\gph\partial \varphi$ and $\gph\partial \varphi_S$, as we alluded to above. Indeed, the subdifferential graph is related to the graph of the proximal map by a linear isomorphism. The guarantee \eqref{eqn:main_result} is also perfectly in line with the recent progress on streaming algorithms \cite{prixm_guide_subgrad,stochastic_subgrad,model_paper_non_euclid,zhang_he}. These works showed that a variety of popular streaming algorithms (e.g. stochastic subgradient, Gauss-Newton, and proximal point) drive the gradient of the Moreau envelope to zero at a controlled rate. Consequently, the estimate \eqref{eqn:main_result} provides a 
tight threshold beyond which such streaming algorithms on the regularized ERM problem \eqref{eq:empirical0} should terminate. The proof we present of \eqref{eqn:main_result} uses only the most elementary techniques: stability of quadratically regularized ERM \cite{stocH-shal}, McDiarmid's inequality \cite{mcdiarmid}, and a covering argument.

It is intriguing to ask when the dimension dependence in the bound \eqref{eqn:main_result} can be avoided. For example, for certain types of losses (e.g. modeling a linear predictor) there are well-known dimension independent bounds on uniform convergence in function values. Can we therefore obtain dimension independent bounds in similar circumstances, but on the deviations $\|\nabla \varphi_{1/2\rho}-\nabla (\varphi_S)_{1/2\rho}\|_2$?
The main tool we use to address this question  is entirely deterministic. We will show that if $\varphi$ and $\varphi_S$ are uniformly $\delta$ close, then the gradients $\nabla \varphi_{1/2\rho}$ and $\nabla(\varphi_S)_{1/2\rho}$ are uniformly $O(\sqrt{\delta})$ close, as well as their subdifferential graphs in the Hausdorff distance.  We illustrate the use of such bounds with two examples.
As the first example, consider the loss $f$ modeling a generalized linear model:
$$f(x, z) = \ell(\dotp{x, \phi(z)}, z).$$ Here $\phi$ is some feature map and $\ell(\cdot, z)$ is a loss function. It is well-known that if  $\ell(\cdot, z)$ is Lipschitz, then the empirical function values $f_S(x)$ converge uniformly to the population values $f(x)$ at a dimension-independent rate that scales as $m^{-1/2}$ in the sample size.  We thus deduce 
that the gradient $\nabla (\varphi_S)_{1/2\rho}$ converges uniformly to $\nabla \varphi_{1/2\rho}$ at the rate $m^{-1/4}$.   We leave it as an intriguing open question whether this rate can be improved to $m^{-1/2}$.
The second example analyzes the landscape of a robust nonlinear regression problem, wherein we observe a series of nonlinear measurements $\sigma(\dotp{\bar x, z})$ of input data $\bar x$, possibly with adversarial corruption. Using the aforementioned techniques, we will show that under mild distributional assumptions on $z$, every stationary point of the associated nonsmooth nonconvex empirical risk is within a small ball around $\bar x$.
	
	Though, in the discussion above, we focused on the norm that is induced by an inner product, the techniques we present apply much more broadly to Bregman divergences. In particular, any Bregman divergence generates an associated conjugate regularization of the empirical and population risks, making our techniques applicable under non-euclidean geometries and under high order growth of the loss function. The outline of the paper is as follows. In Section~\ref{sec:prelim}, we introduce our notation as well as several key concepts including Bregman divergences and subdifferentials. In Section~\ref{sec:prob_sett}, we describe the problem setting. In Section~\ref{sec:dimension_dep}, we  obtain dimension dependent rates on the error between the gradients of the Moreau envelopes of the population and subsampled objectives. In Section~\ref{sec:dim_indep}, we illustrate the techniques of the previous section by obtaining dimension independent rates for generalized linear models and analyzing the landscape of a robust nonlinear regression problem.

	\subsection*{Related literature.}		
	
	This paper builds on the vast literature on sample average approximations
	found in the stochastic programming and statistical learning literature. The results in these communities are similar in many respects, but differ in their focus on convergence criteria. In the stochastic programming literature, much attention has been given to the convergence of (approximate) minimizers and optimal values both in the distributional and almost sure limiting sense~\cite{rockafellar_asymptotic_normality,shapiro_monte_carlo,geyer1994,shapiro2000_local_m_estimation,convergence_of_local_min,probability_metrics,Shapiro2000Stochastic,Kaniovski1995,stability_wets}. In contrast, the statistical learning community puts a greater emphasis on excess risk bounds that hold with high probability, often with minimal or no dependence on dimension~\cite{stocH-shal,grunwald2016fast,NIPS2010_3894,liu2018fast,mehta2016fast,Mehta:2014:SMF:2968826.2968960,stability_baquett,doi:10.1142/S0219530505000650,JMLR:v16:vanerven15a,Zinkevich:2003:OCP:3041838.3041955,NIPS2008_3400,kakade2009complexity}. 
	
	Several previous works have studied (sub)gradient based convergence, as we do here. For example,~\cite{XU2010692} proves nonasymptotic, dimension dependent high probability bounds on the distance between the empirical and population subdifferential under the Hausdorff metric. The main assumption in this work, however, essentially requires smoothness of the population objective. The work~\cite{Xu2009} takes a different approach, directly smoothing the empirical losses $f(x, z)$. They show that the limit of the gradients of a certain smoothing of the empirical risk converges to an element of the population subdifferential. No finite-sample bounds are developed in \cite{Xu2009}. The most general asymptotic convergence result that we are aware of is presented in~\cite{SHAPIRO20071390}. There, the authors show that with probability one, the limit of a certain enlarged subdifferential of the empirical loss converges to an enlarged subdifferential of the population risk under the Hausdorff metric. 
	
	The two works most closely related to this paper are more recent. The paper~\cite{landscape_grad_mont} proves high probability uniform convergence of gradients for smooth objectives under the assumption that the gradient $\nabla f(x,z)$ is sub-Gaussian with respect to the population data. The bounds presented in~\cite{landscape_grad_mont} are all dimension dependent and rely on covering arguments. The more recent paper~\cite{karthik_new}, on the other hand, provides dimension independent high probability uniform rates of convergence of gradients for smooth Lipschitz generalized linear models. The main technical tool developed~\cite{karthik_new} is a ``chain rule" for Rademacher complexity. We note that, in contrast to the $m^{-1/4}$ rates developed in this paper,~\cite{karthik_new} obtains rates of the form $m^{-1/2}$ for smooth generalized linear models.

\section{Preliminaries.}\label{sec:prelim}
Throughout, we follow standard notation from convex analysis, as set out for example by Rockafellar \cite{con_ter}.  
The symbol $\R^d$ will denote a $d$-dimensional Euclidean space with inner product $\langle\cdot,\cdot \rangle$ and the induced norm $\|x\|_2=\sqrt{\langle x, x\rangle}$. Given any other norm $\|\cdot\|$, the symbol $\|\cdot\|_*$ will denote the dual norm. For any set $Q\subset\R^d$, we let $\inter\, Q$ and $\cl Q$  stand for the interior and closure of $Q$, respectively. The symbol $\ri Q$ will refer to the interior of $Q$ relative to its affine hull.
The closed Euclidean unit ball and the unit simplex in $\R^d$ will be denoted by $\bf B$ and $\Delta$, respectively.
The effective domain of any function $f\colon\R^d\to\R\cup\{\infty\}$, denoted by $\dom f$,  consists of all points where $f$ is finite. The indicator function of any set $Q\subset\R^d$, denoted $\iota_{Q}$, is defined to be zero on $Q$ and $+\infty$ off it.

\subsection{Bregman divergence.}
In this work, we will use techniques based on the Bregman divergence, as is now standard in optimization and machine learning literature. For more details on the topic, we refer the reader to the expository articles \cite{bubeck2015convex,juditsky_NEMfirs,Teboulle2018}. 
To this end, henceforth, we fix a {\em Legendre function} $\Phi\colon\R^d\to\R\cup\{\infty\}$, meaning:
\begin{enumerate}
	\item (Convexity) $\Phi$ is proper, closed, and strictly convex.
	\item (Essential smoothness) The domain
	of $\Phi$ has nonempty interior,	$\Phi$ is differentiable on $\inter(\dom \Phi)$ and 
	for any sequence $\{x_k\}\subset \inter(\dom \Phi)$ converging to a boundary point of $\dom \Phi$, it must be the case that $\|\nabla \Phi(x_k)\|_2\to\infty$.
\end{enumerate}
Typical examples of Legendre functions are the squared Euclidean norm $\Phi(x)=\frac{1}{2}\|x\|^2_2$, convex  polynomials $\Phi(x)=\|x\|_2+\|x\|^r_2$, the Shannon entropy $\Phi(x)=\sum_{i=1}^d x_i\log(x_i)$ with $\dom \Phi=\R^d_+$, and the Burge function $\Phi(x)=-\sum_{i=1}^d\log(x_i)$ with $\dom \Phi= \R^d_{++}$. For more examples, see the articles \cite{AusTeb,Baus_bor_breg,Eck_breg,teb_breg_2} or the recent survey \cite{Teboulle2018}.  

The function $\Phi$ induces the {\em Bregman divergence} defined by
\[ D_\Phi (y,x) := \Phi(y) - \Phi(x) - \la \nabla \Phi (x), y-x \ra,\]
for all $x \in \inter(\dom \Phi),~y\in \dom \Phi$.  In the classical setting $\Phi(\cdot)=\tfrac{1}{2}\|\cdot\|^2_2$, the divergence reduces to the square deviation $D_\Phi (y,x)=\frac{1}{2}\|y-x\|^2_2$.
Notice that since $\Phi$ is strictly convex, equality $D_\Phi (y,x)=0$ holds for some $x,y\in \inter(\dom \Phi)$ if and only if $y=x$. Therefore, the quantity $D_\Phi (y,x)$ measures the proximity between the two points $x,y$.
The divergence typically is asymmetric in $x$ and $y$, and therefore we define the symmetrization 
\[ \bregsym(x, y):=D_\Phi (x,y)+D_\Phi (y,x).\]

A function $g$ is called {\em compatible} with $\Phi$ if  the inclusion $\ri(\dom g)\subset\inter(\dom \Phi)$ holds. Compatibility will be useful for guaranteeing that ``proximal points'' of $g$ induced by $\Phi$ lie in $\inter(\dom \Phi)$; see the forthcoming Theorem~\ref{thm:env_diff}.
Our focus will be primarily on those functions that can be convexified by adding a sufficiently large multiple of $\Phi$. Formally, we will say that a function $g\colon\R^d\to\R\cup\{\infty\}$ is {\em $\rho$-weakly convex relative to $\Phi$}, for some $\rho\in \R$, if the perturbed function $x\mapsto g(x)+\rho\Phi(x)$ is convex. Similarly $g$ is {\em $\alpha$-strongly convex relative to $\Phi$}  if the function $x\mapsto g(x)-\alpha\Phi(x)$ is convex. The notions of relative weak/strong convexity are closely related to the generalized descent lemma introduced in \cite{descentBBT} and the recent  work on relative smoothness in \cite{rel_smooth_freund,Lu_mirror_weird}. We postpone discussing examples of weakly convex functions to section~\ref{sec:prob_sett}.

\subsection{Subdifferentials.}
First-order optimality conditions for nonsmooth and nonconvex problems are often most succinctly stated using subdifferentials.
The {\em subdifferential} of a function $g$ at a point $x\in \dom g$ is denoted by $\partial g(x)$ and consists of all vectors $v\in\R^d$ satisfying
$$g(y)\geq g(x)+\langle v,y-x\rangle+o(\|y-x\|)\qquad \textrm{as }y\to x.$$
When $g$ is differentiable at $x$, the subdifferential reduces to the singleton $\partial g(x)=\{\nabla g(x)\}$, while for convex functions it reduces to the subdifferential in the sense of convex analysis.
We will call a point $x$ {\em critical} for $g$ if the inclusion $0\in \partial g(x)$ holds.

When $g$ is $\rho$-weakly convex relative to $\Phi$, the subdifferential automatically satisfies the seemingly stronger property \cite[Lemma 2.2]{model_paper_non_euclid}:
\begin{equation}\label{eqn:lower_bound_breg}
	g(y)\geq g(x)+\langle v,y-x\rangle-\rho D_{\Phi}(y,x).
\end{equation}
for any $x\in \inter(\dom \Phi),y\in \dom \Phi$, and any $v\in\partial g(x)$. It is often convenient to interpret the assignment $x\mapsto \partial g(x)$ as a set-valued map, and as such, it has a graph defined by 
$$\gph\partial g(x):=\{(x,v)\in \R^d\times\R^d: v\in \partial g(x)\}.$$

\section{Problem setting.}\label{sec:prob_sett}
Fix a probability space $(\Omega, \mathcal{F}, P)$.
In this paper, we focus on the optimization problem 
\begin{align}\label{eq:population_sect}
	\min_{x\in \R^d} \; \varphi(x)=f(x) +r(x)\qquad \textrm{ where }\qquad  f(x)=\EE_{z \sim P}\left[ f(x, z) \right],
\end{align}
under the following assumptions on the functional components: 
\begin{enumerate}
	\item[(A1)] (\textbf{Bregman Divergence}) The Legendre function $\Phi\colon\R^d\to\R\cup\{\infty\}$ satisfies the compatibility condition, $\ri(\dom r)\subset \inter(\dom \Phi).$
	Moreover $\Phi$ is $\alpha$-strongly convex with respect to some norm $\|\cdot \|$ on the set $\dom r\cap \inter(\dom \Phi)$, meaning: 
	$$\frac{\alpha}{2}\|y-x\|^2\leq D_{\Phi}(x,y)\qquad \textrm{ for all }x,y\in \dom r\cap \inter(\dom \Phi).$$
	
	\item[(A2)] (\textbf{Weak Convexity}) The functions $f(\cdot,z)+r(\cdot)$ are $\rho$-weakly convex relative to $\Phi$, for a.e. $z\in \Omega$, and are bounded from below.
	\item[(A3)] (\textbf{Lipschitzian Property}) There exists a square integrable function $L\colon \Omega \rightarrow \RR_+$ such that for all $x, y \in \inter(\dom \Phi) \cap \dom r$ and $z \in \Omega$, we have
	\begin{align*}
		|f(x, z) - f(y, z)| &\leq L(z)\sqrt{\bregsym(x, y)},\\
		\sqrt{\EE\left[ L(z)^2 \right]}&\leq \sigma.
	\end{align*}
	
\end{enumerate}

The stochastic optimization problem \eqref{eq:population_sect} is the standard task of minimizing the regularized population risk. The function $f(x,z)$ is called the loss, while  $r\colon\R^d\to\R\cup\{\infty\}$ is a structure promoting regularizer. Alternatively $r$ can encode feasibility constraints as an indicator function. Assumptions $(A1)$ and $(A2)$ are self-explanatory. Assumption $(A3)$ asserts control on the variation in the loss $f(\cdot,z)$. In the simplest setting when $\Phi=\frac{1}{2}\|\cdot\|^2_2$, Assumption $(A3)$ simply amounts to Lipschitz continuity of the loss $f(\cdot,z)$ on $\dom r$ with a square integrable Lipschitz constant $L(z)$. The use of the Bregman divergence allows much more flexibility, as highlighted in the recent works \cite{descentBBT,rel_smooth_freund,Lu_mirror_weird,model_paper_non_euclid}. On one hand, it allows one to consider losses that are Lipschitz in a non-Euclidean norm, as long as $\Phi$ is strongly convex in that norm---a standard technique in machine learning and optimization. On the other hand, the Bregman divergence could also accommodate losses that are only locally Lipschitz continuous. The following two examples illustrate these two settings.

\begin{example}[Non-Euclidean geometry: $\ell_1$-setup]\label{ex:non_euclid}{\rm 
	In this example, we consider the instance of \eqref{eq:population_sect} that is adapted to the $\ell_1$-norm $\|\cdot\|_1$. To this end, define the Legendre function $\Phi(x)=\frac{e^2}{2(p-1)}\|x\|^2_p$ with $p=1+\frac{1}{\ln d}$. Then $\Phi$ is strongly convex with respect to the $\ell_1$-norm with constant $\alpha:=1$; see e.g. \cite[Lemma 17]{shaithesis}, \cite[eqn. 5.5]{complexity}, \cite[Corollary 2.2]{nest_men_acta}.  Then  $(A2)$  holds as long as the functions $g_z(y):=f(y,z)+r(y)$ satisfy the estimate \cite[Lemma 2.2]{model_paper_non_euclid}
	$$g_z(y)\geq g_z(x)+\langle v,y-x\rangle -\frac{\rho }{2}\|y-x\|^2_1,$$
	for all $x,y\in \dom g$, $v\in \partial g_z(x)$, and a.e. $z\in \Omega$.   Property $(A3)$, in turn, holds as long as the functions $f(\cdot,z)$ are Lipschitz continuous in the $\ell_1$-norm, with a square integrable Lipschitz constant $L(z)$.}
\end{example}
\smallskip

\begin{example}[High-order growth]
	{\rm
	As the second example, we consider the setting where the losses grow faster than linear. Namely, suppose that the Lipschitz constant of the loss on bounded sets is polynomially controlled:
	\begin{align*}
		\frac{f(x,z) - f(y, z)}{\|x-y\|_2} \leq  L(z) \sqrt{\frac{p(\|x\|_2)+p(\|y\|_2)}{2}} \qquad \textrm{ for all distinct } x,y\in\R^d,z\in \Omega,
	\end{align*}
	where $p(u)=\sum_{i=0}^r a_i u^i$ is some degree $r$ univariate polynomial  
	with nonnegative coefficients and $L(\cdot)$ is square integrable. Then the result \cite[Proposition 3.2]{model_paper_non_euclid} shows that (A3) holds for the Legendre function 
	\begin{equation}
		\Phi(x) := \sum_{i=0}^r a_i\left(\frac{3i + 7}{i+2} \right)\|x\|_2^{i+2}.
	\end{equation}}
\end{example}

We could in principle state and prove all the results of the paper in the Euclidean setting $\Phi=\frac{1}{2}\|\cdot\|^2_2$. On the other hand, the use of the Bregman divergence adds no complications whatsoever, and therefore we work in this greater generality. The reader should keep in mind, however, that all the results are of interest even in the Euclidean setting.

\subsection{Convex compositions.}
The most important example of the problem class \eqref{eq:population_sect} corresponds to the setting when $r(\cdot)$ is convex and the loss has the form:
$$f(x,z)=h(c(x,z),z),$$
where $h(\cdot,z)$ is convex and $c(\cdot,z)$ is $C^1$-smooth. To see this, let us first consider the setting $\Phi=\frac{1}{2}\|\cdot\|^2_2$. Then provided that $h(\cdot,z)$ is $\ell$-Lipschitz and the Jacobian $\nabla c(\cdot,z)$ is $\beta$-Lipschitz, a quick argument \cite[Lemma 4.2]{comp_DP} shows that the loss $f(\cdot,z)$ is $\ell\beta$-weakly convex relative to $\Phi$; therefore (A2) holds with $\rho=\ell\beta$. Moreover, if there exists a square integrable function $M(\cdot)$ satisfying $\|\nabla c(x,z)\|_{\rm op}\leq M(z)$ for all $x\in \dom r$ and $z\in \Omega$, then (A3) holds with $L(z)=\ell M(z)$.

The assumption that $h(\cdot,z)$ is Lipschitz continuous is mild and is often true in applications. On the other hand, Lipschitz continuity and boundedness of $\nabla c(\cdot,z)$ are strong assumptions. They can be relaxed by switching to a different Bregman divergence. Indeed, suppose that  $h(\cdot,z)$ is convex and $\ell$-Lipschitz as before, while $c(\cdot,z)$ now satisfies the two growth properties:
\begin{align*}
	\frac{\|\nabla c(x,z) - \nabla c(y,z) \|_{{\rm op}}}{\|x - y\|_2} &\leq p(\|x\|_2) + p(\|y\|_2)\qquad \forall x\neq y, \forall z\in \Omega\\
	\|\nabla c(x,z)\|_{\rm op} &\leq L_1(z) \cdot \sqrt{q(\|x\|_2)}\qquad \forall x, \forall z\in\Omega,
\end{align*}
for some polynomials $p(u)=\sum_{i=}^r a_i u^i$ and $q(u)=\sum_{i=}^r b_i u^i$, with $a_i,b_i\geq 0$, and some square integrable function $L_1(\cdot)$. Define the Legendre function
$$\Phi(x)=\sum_{i = 0}^r \frac{b_i+a_i(3i + 7)}{i+2}\|x\|_2^{i+2}.$$
Then the result \cite[Propositions 3.4, 3.5]{model_paper_non_euclid} shows that assumption (A2) and (A3) hold with $\rho=\ell$ and $L(z)=\sqrt{2}\ell L_1(z)$.  

The class of composite problems is broad and has attracted some attention lately \cite{noncon_rob_rec,weirdn2,duchi_ruan,davis2017nonsmooth,comp_DP,duchi_ruan_PR,stochastic_subgrad,zhang_he,model_paper_non_euclid,prixm_guide_subgrad} as  an appealing setting for nonsmooth and nonconvex optimization. Table~\ref{table:tab1} summarizes a few interesting problems of this type; details can be found in the aforementioned works.

\begin{table}[h!]
	\begin{center}
		\begin{tabular}{  | l | c | c | }
			\hline
			{\bf Problem} & {\bf Loss function} & {\bf Regularizer} \\ \hline\hline
			Phase retrieval & $f(x,(a,b))=|\langle a,x\rangle^2-b|$ & $r(x)=0,\|x\|_1$ \\ \hline
			Blind deconvolution & $f((x,y),(u,v,b))=| \langle u,x\rangle \langle v, y\rangle-b|$ & --- \\  
			\hline
			Covariance estimation & $f(x,(U,b))=| \|Ux\|^2-b|$ & --- \\  
			\hline
			Censored block model & $f(x,(ij,b))=|x_ix_j-b|$ & --- \\  
			\hline
			Conditional Value-at-Risk & $f((x,\gamma),z)=(\ell(x,z)-\gamma)^+$ & $r(x,\gamma)= (1-\alpha)\gamma$ \\  
			\hline
			Trimmed estimation & $f((x,w),i)=w_if_i(x)$ & $r(x,w)=\iota_{[0,1]^n\cap k\Delta}(w)$ \\  
			\hline
			Robust PCA & $f((U,V),(ij,b))=|\langle u_i,v_j\rangle-b|$ & ---\\  
			\hline
			Sparse dictionary learning & $f((D,x),z)=\|z-Dx\|_2$ & $r(D,x)=\iota_{\bf B}(D)+\lambda\|x\|_1$\\  
			\hline		
		\end{tabular}
		\caption{Common stochastic weakly convex optimization problems.}
		\label{table:tab1}
	\end{center}
\end{table}

\subsection{The stationarity measure and implicit smoothing.}
Since the problem \eqref{eq:population_sect} is nonconvex and nonsmooth, typical algorithms can only be guaranteed to find critical points of the problem, meaning those satisfying $0\in \partial \varphi(x)$. Therefore, one of our main goals is to estimate the Hausdorff distance between the subdifferential graphs, $\gph \partial \varphi$ and $\gph \partial \varphi_S$. We employ an indirect strategy based on a smoothing technique. 

Setting the formalism, for any function $g\colon\R^d\to\R\cup\{\infty\}$, we define the {\em $\Phi$-envelope} and the {\em $\Phi$-proximal map}:
\begin{align*}
	g^\Phi_{\lambda}(x)&:=\inf_{y}\, \left\{g(y)+\frac{1}{\lambda}D_{\Phi}(y,x)\right\},\qquad
	\prox_{\lambda g}^\Phi(x):=\argmin_{y}\, \left\{g(y)+\frac{1}{\lambda}D_{\Phi}(y,x)\right\},
\end{align*}
respectively. 
Note that in the Euclidean setting $\Phi=\frac{1}{2}\|\cdot\|^2_2$, these two constructions reduce to the widely used Moreau envelope and the proximity map introduced in \cite{MR0201952}. In this case, we will omit the subscript $\Phi$ from $g^\Phi_{\lambda}$ and $\prox_{\lambda g}^\Phi$. Note that in the Euclidean setting, the graphs of the proximal map and the subdifferential are closely related:
$$y=\prox_{\lambda g}(x)\quad\Longleftrightarrow\quad (y, \lambda^{-1}(x-y))\in \gph\partial g.$$
Consequently the graph of the proximal map $\prox_{\lambda g}$ is linearly isomorphic to the graph of the subdifferential $\partial g$ by the linear map $(x,y)\mapsto (y,\lambda^{-1}(y-x))$. It is this observation that will allow us to pass from uniform estimates on the deviations $\| \prox_{\varphi/\lambda}(x)-\prox_{\varphi_S/\lambda}(x)\|$ to estimates on the Hausdorff distance between subdifferential graphs, $\gph \partial \varphi$ and $\gph \partial \varphi_S$.

It will be important to know when the set $\prox_{\lambda g}^\Phi(x)$ is a singleton lying in $\inter(\dom \Phi)$.  
The following theorem follows quickly from \cite{breg_baus}, with a self-contained argument given in \cite[Theorem 4.1]{model_paper_non_euclid}. The Euclidean version of the result was already proved in \cite{MR0201952}. We will often appeal to this theorem without explicitly referencing it to shorten the exposition.

\begin{thm}[Smoothness of the $\Phi$-envelope] \label{thm:env_diff}
	Consider a closed and lower-bounded function $g\colon\R^d\to\R\cup\{\infty\}$ that is $\rho$-weakly convex and compatible with some Legendre function $\Phi\colon\R^d\to\{\infty\}$. Then the following are true for any positive $\lambda< \rho^{-1}$ and any $x\in \inter(\dom \Phi)$. 
	\begin{itemize}
		\item The proximal point $\prox_{\lambda g}^\Phi(x)$ is uniquely defined and lies in $\dom g\cap \inter(\dom \Phi)$.
		\item If $\Phi$ is twice differentiable on the interior of its domain, then 
		the envelope $g_\lambda^\Phi$ is differentiable at $x$ with gradient given by
		\begin{equation}\label{eqn:grad_eqn}
			\nabla g^\Phi_\lambda (x) := \frac{1}{\lambda} \nabla^2 \Phi (x) \left(x - \prox_{\lambda g}^\Phi (x)\right).
		\end{equation}
	\end{itemize}
\end{thm}

The main application of Theorem~\ref{thm:env_diff} is to the functions $\varphi$ and $\varphi_S$ under the assumptions (A1)-(A3). Looking at the expression \eqref{eqn:grad_eqn}, it is clear that the deviation between $x$ and $\prox_{\lambda g}^\Phi(x)$ provides an estimate on the size of the gradient $\nabla g^\Phi_{\lambda}(x)$.
To make this observation precise, for any point $x\in \inter(\dom \Phi)$, define the primal-dual pair of local norms:
$$\|y\|_x:=\left\|\nabla^2 \Phi(x)y\right\|_*,\qquad \|v\|_x^*=\left\|\nabla^2 \Phi(x)^{-1}v\right\|.$$
Thus for any positive $\lambda< \rho^{-1}$ and $x\in \inter(\dom \Phi)$, using \eqref{eqn:grad_eqn}, we obtain the estimate
$$\sqrt{D_{\Phi}\left(\prox_{\lambda g}^\Phi (x),x\right)}\geq \lambda\sqrt{\frac{\alpha}{2}}\cdot\|\nabla g^\Phi_\lambda (x)\|^*_x.$$
Consequently, following the recent literature on the subject, we will treat the quantity
\begin{equation}\label{eqn:breg_dist}
	D_{\Phi}(\prox_{\lambda g}^\Phi(x),x),
\end{equation}
as a measure of approximate stationarity of $g$ at $x$. In particular, when specializing to the target setting $g=\varphi$, it is this quantity \eqref{eqn:breg_dist} that streaming algorithms drive to zero at a controlled rate, as shown in \cite{model_paper_non_euclid,stochastic_subgrad,zhang_he}.

\section{Dimension Dependent Rates.}\label{sec:dimension_dep}
In this section, we prove the uniform convergence bound~\eqref{eqn:main_result}, appropriately generalized using the Bregman divergence. The proof outline is as follows. First, in Theorem~\ref{thm:stab_ERM} we will estimate the expected error between the true proximal point and an empirical sample, 
$$\EE_S\|\prox^{\Phi}_{\varphi/\lambda}(y)-\prox^{\Phi}_{\varphi_S/\lambda}(y)\|,$$
where $y$ is fixed. The same theorem also shows $\prox^{\Phi}_{\varphi_S/\lambda}(y)$ is stable in the sense that  $\prox^{\Phi}_{\varphi_S/\lambda}(y)$ does not vary much when one sample $z_i\in S$ is changed arbitrarily. In the Euclidean setting, this is precisely the main result of \cite{stocH-shal} on stability of quadratically regularized ERM in stochastic convex optimization. Using McDiarmid's inequality \cite{mcdiarmid} in Theorem~\ref{thm:prox_close_p}, we will then  deduce that the quantity $\|\prox^{\Phi}_{\varphi/\lambda}(y)-\prox^{\Phi}_{\varphi_S/\lambda}(y)\|$ concentrates around its mean for a fixed $y$. A covering argument over the points $y$ in an appropriate metric will then complete the proof.

We begin following the outlined strategy with the following theorem. 

\begin{thm}[Stability of regularized ERM]\label{thm:stab_ERM} Consider a set $S=(z_1,\ldots,z_m)$ and define $S^i := (z_1, \ldots, z_{i-1}, z_i', z_{i+1}, \ldots z_m)$, where both the index $i$ and the point $z_i' \in \Omega$ are arbitrary. Fix an arbitrary point $y\in \inter(\dom \Phi)$ and a real $\bar \rho>\rho$, and set 
	\begin{align*}
	\alg^* := \argmin_{y \in \domainspace}\,\left\{\varphi(y)+\bar\rho D_{\Phi}(x,y)\right\} && \text{ and } &&  \alg(S) := \argmin_{y \in \domainspace} \,\left\{\varphi_S(y)+\bar\rho D_{\Phi}(x,y)\right\}.
	\end{align*}
	Then  the estimates hold:
	\begin{align}
	\sqrt{\bregsym(\alg(S),\alg(S^i))}&\leq \frac{L(z_i)+L(z_i')}{(\bar\rho-\rho)m}\label{eqn:stab_guarant}\\
	\EE_S \left[D_{\Phi}(\alg(S),\alg^*)\right]&\leq \frac{2\sigma^2}{(\bar\rho-\rho)^2 m}\label{eqn:comparsiontoopt}\\
	0\leq\EE_S[\varphi^{\Phi}_{1/\bar\rho}(y)-(\varphi^{\Phi}_S)_{1/\bar\rho}(y)] &\leq \frac{2\sigma^2}{(\bar\rho-\rho)m}.\label{conv_est_phi}
	\end{align}
\end{thm}
\begin{proof}
	We first verify~\eqref{eqn:stab_guarant}. A quick computation yields for any points $u$ and $v$ the equation:
	\begin{equation}\label{eqn:rewrite_eqn}
	f_S(v)-f_S(u)=f_{S^{i}}(v)-f_{S^{i}}(u)+\frac{f(v,z_i)-f(u,z_i)}{m}+\frac{f(u,z_i')-f(v,z_i')}{m}.
	\end{equation}
	Define now the regularized functions
	$$\widehat{\varphi}(x) := \varphi(x) + \bar\rho D_{\Phi}(x,y) \qquad \text{ and } \qquad \widehat{\varphi}_{S}(x) := \varphi_S(x) + \bar\rho D_{\Phi}(x,y).$$
	Then adding $[r(v)+\bar\rho D_{\Phi}(v,y)]-[r(u)+\bar\rho D_{\Phi}(u,y)]$ to both sides of \eqref{eqn:rewrite_eqn}, we obtain
	\begin{align*}
	\widehat{\varphi}_S(v)-\widehat{\varphi}_S(u)=\widehat{\varphi}_{S^i}(v)-\widehat{\varphi}_{S^i}(u)+\frac{f(v,z_i)-f(u,z_i)}{m}+\frac{f(u,z_i')-f(v,z_i')}{m}.
	\end{align*}
	Henceforth, set $v:=\alg(S^i)$ and $u:= \alg(S).$
	Thus $v$ is the minimizer of $\widehat{\varphi}_{S^i}$ and $u$ is the minimizer of $\widehat{\varphi}_{S}$.
	Taking into account that  $\widehat{\varphi}_S(\cdot)$  and $\widehat{\varphi}_{S^i}(\cdot)$ are $(\bar\rho-\rho)$-strongly convex relative to $\Phi$, we deduce
	$$(\bar\rho-\rho)D_{\Phi}(v,u)\leq \widehat{\varphi}_S(v)-\widehat{\varphi}_S(u)\leq \frac{f(v,z_i)-f(u,z_i)}{m}+\frac{f(u,z_i')-f(v,z_i')}{m}-(\bar\rho-\rho)D_{\Phi}(u,v).$$
	Rearranging, we arrive at the estimate
	\begin{align*}
	\bregsym(u,v)&\leq \frac{1}{\bar\rho-\rho}\left[\frac{f(v,z_i)-f(u,z_i)}{m}+\frac{f(u,z_i')-f(v,z_i')}{m}\right]\leq \frac{L(z_i)+L(z_i')}{(\bar\rho-\rho)m}\sqrt{\bregsym(u,v)}.
	\end{align*}
	Dividing through by $\sqrt{\bregsym(u,v)}$, we obtain the claimed stability guarantee \eqref{eqn:stab_guarant}.
	
	To establish \eqref{conv_est_phi}, observe first 
	$$(\varphi_S)_{1/\bar\rho}(y)=\varphi_S(\alg(S))+\bar\rho D_{\Phi}(\alg(S),y)\leq \varphi_S(x)+\bar\rho D_{\Phi}(x,y)\qquad \textrm{for all }x.$$ Taking expectations, we conclude $\EE_S[(\varphi_S)_{1/\bar\rho}(y)]\leq  \varphi_{1/\bar\rho}(y)$, which is precisely the left-hand-side of \eqref{conv_est_phi}. Next, it is standard to verify the expression \cite[Theorem 13.2]{shalev2014understanding}:
	\begin{equation}
	\begin{aligned}
	\EE_S[f(A(S))]&=\EE_S[f_S(A(S))]+\EE_S[f(A(S))-f_S(A(S))]\\
	&= \EE_S[f_S(A(S))]+\EE_{(S,z')\sim P,i\sim U(m)} [f(A(S^i),z_i)-f(A(S),z_i)],
	\end{aligned}
	\end{equation}
	where $U(m)$ denotes the discrete uniform distribution.
	Taking into account \eqref{eqn:stab_guarant}, we obtain 
	\begin{align}
	\left|\EE_S[\widehat\varphi(A(S))-\widehat\varphi_S(A(S))]\right|&\leq \EE\left[L (z)\sqrt{\bregsym(\alg(S),\alg(S^i))}\right]\notag\\
	&\leq \sqrt{\EE_{z}[L (z)^2]}\sqrt{\EE_{S}[\bregsym(\alg(S),\alg(S^i)]}\leq \frac{2\sigma^2}{(\bar\rho-\rho)m}.\label{eqn:key_eqn_abs}
	\end{align}
	Noting $\widehat\varphi(A(S))\geq \varphi^{\Phi}_{1/\bar\rho}(y)$ and $(\varphi^{\Phi}_S)_{1/\bar\rho}(y)=\widehat\varphi_S(A(S))$  yields the right-hand-side of \eqref{conv_est_phi}.
	
	Finally taking into account that $\widehat{\varphi}$ is $(\bar\rho-\rho)$-strongly convex relative to $\Phi$, we deduce 
	$$(\bar\rho-\rho)D_{\Phi}(\alg(S),\alg^*)\leq \widehat{\varphi}(\alg(S))-\min \widehat{\varphi}   .$$
	Taking expectation, and using the inequalities $\EE_S[\widehat\varphi_S(A(S))]\leq \min  \widehat\varphi$ and \eqref{eqn:key_eqn_abs},  we arrive at 
	$$(\bar\rho-\rho)\EE_S[D_{\Phi}(\alg(S),\alg^*)]\leq \EE_S[\widehat{\varphi}(\alg(S))-\min \widehat{\varphi}]\leq \EE_S[\widehat{\varphi}(\alg(S))- \widehat{\varphi}_S(\alg(S))] \leq\frac{2\sigma^2}{(\bar\rho-\rho)m}.$$
	Thus we have established \eqref{eqn:comparsiontoopt}, and the proof is complete. 
\end{proof}

Next, we pass to high probability bounds on the deviation 
$\|\prox_{\varphi/\lambda}(y)-\prox_{\varphi_S/\lambda}(y)\|$
by means of McDiarmid's inequality \cite{mcdiarmid}. The basic result reads as follows. Suppose that a function $g$ satisfies the bounded difference property: \begin{align*}
|g(z_1,\ldots,z_{i-1},z_i,z_{i+1},\ldots,z_m) - g(z_1,\ldots,z_{i-1},z_i',z_{i+1},\ldots,z_m)| \leq  c_i,\end{align*}
for all $i, z_1,\ldots,z_{i-1},z_i,z_{i+1},\ldots,z_m,z_i'$,
where $c_i$ are some constants. Then for any independent random variables $x_1,\ldots, x_m$, the random variable $Y=g(x_1,\ldots, x_m)$ satisfies:
$$\mathbb{P}(Y-\EE Y\geq t)\leq \exp\left(\frac{-2t^2}{\|c\|_2^2}\right).$$
A direct application of this inequality to $\|\prox^{\Phi}_{\varphi/\lambda}(y)-\prox^{\Phi}_{\varphi_S/\lambda}(y)\|$ using \eqref{eqn:stab_guarant} would require the Lipschitz constant $L(\cdot)$ to be globally bounded. This could be a strong assumption, as it essentially requires the population data to be bounded. We will circumvent this difficulty by extending the McDiarmid's inequality to the setting where the constants $c_i$ are replaced by some functions $\omega_i(\cdot,\cdot)$ of the data, $z_i$ and $z_i'$. Let $\varepsilon_i$ be a Rademacher 
random variable, meaning a random random variable taking value $\pm 1$ with equal probability. 
Then as long as the symmetric random variables $\varepsilon_i\omega_i(z_i,z'_i)$ have sufficiently light tails, a McDiarmid type bound will hold. In particular, we will be able to derive high probability upper bounds on the deviations $\|\prox_{\varphi/\lambda}^{\Phi}(y)-\prox_{\varphi_S/\lambda}^{\Phi}(y)\|$ only assuming that the random variable $\varepsilon L$ is sub-Gaussian. The proof follows known techniques for establishing McDiarmid's inequality, and in particular is essentially the same as that in~\cite[Theorem 1]{kontorovich2014concentration}, though there the statement of the theorem assumed that $\omega_i$ is a metric and $\varepsilon_i\omega(z_i,z_i)$ is sub-Gaussian.

Henceforth, given a random variable $X$, we will let $\psi(\lambda):=\log(\EE e^{\lambda X})$ denote the logarithm of the moment generating function. The symbol $\psi^{\star}\colon\R\to\R\cup\{\infty\}$ will stand for the Fenchel conjugate of  $\psi$, defined by
$\psi^{\star}(t)=\sup_{\lambda} \{t\lambda -\psi(\lambda)\}$.

\begin{thm}[McDiarmid extended]\label{thm:McDiarmid}
	Let $z_1, \ldots, z_m$ be independent random variables with ranges $z_i \in \probspace_i$. Suppose that there exist functions $\omega_i\colon \probspace_i\times \probspace_i \rightarrow \RR_{+}$ such that the inequality \begin{align*}
	|g(z_1,\ldots,z_{i-1},z_i,z_{i+1},\ldots) - g(z_1,\ldots,z_{i-1},z_i',z_{i+1},\ldots)| \leq  \omega_i(z_i,z_i'),
	\end{align*}
	holds for all $z_j\in \probspace_j$, $z_i,z_i'\in \probspace_i$, and all $i,j\in \{1,\ldots, m\}$.
	Then the estimate holds: 
	\begin{equation}\label{eqn:moment_gen}
	\psi_{g(z)-\EE[g(z)]}(\lambda)\leq \sum_{i=1}^m \psi_{\varepsilon_i\omega_i}(\lambda)\qquad \forall \lambda,
	\end{equation}
	where $\omega_i$ denotes the random variable $\omega_i(z_i,z'_i)$ and $\varepsilon_i$ are i.i.d Rademacher random variables.
	In particular if $\omega_i=\omega_j$ for all indices $i$ and $j$, then we have
	\begin{equation}\label{eqn:prox_bound}
	\mathbb{P}(g(z)-\EE[g(z)]\geq t)\leq \exp\left(-m\psi_{\varepsilon\omega}^\star\left(\frac{t}{m}\right)\right)\qquad \forall t\geq 0.
	\end{equation}
\end{thm}
\begin{proof}
	Define the Doob martingale sequence:
	$$Y_0:=\EE[g(z)]\qquad \textrm{and}\qquad Y_i:=\EE[g(z)\mid z_1,\ldots z_{i}]~  \textrm{ for }i=1,\ldots,m,$$
	and consider the martingale differences
	$$V_i:=Y_i-Y_{i-1}\qquad \textrm{for }i=1,\ldots,m.$$ 
	We aim to bound the moment generating function of $Y_m=g(z)$. To this end, observe
	\begin{equation}\label{eqn:induction}
	\EE[e^{\lambda Y_i}]=\EE\left[e^{\lambda Y_{i-1}}\EE[ e^{\lambda (Y_i-Y_{i-1})} \mid z_1,\ldots,z_{i-1}]\right].
	\end{equation}
	Thus, the crux of the proof is to bound the conditional expectation $\EE[ e^{\lambda V_i} \mid z_1,\ldots,z_{i-1}]$. 
	
	Form a vector $z'$  from $z$ by replacing $z_i$ by an identical distributed copy $z_i'$. Clearly then $$\EE[g(z')\mid z_1,\ldots,z_i]= \EE[g(z) \mid z_1,\ldots,z_{i-1}]=Y_{i-1}.$$
	Therefore we may write
	$V_i=Y_i-Y_{i-1}=\EE[g(z')-g(z)\mid z_1,\ldots,z_i]$. Hence, we deduce
	\begin{align*}
	\EE\left[ e^{\lambda V_i} \mid z_1, \ldots, z_{i-1}\right]
	&= \EE\left[ e^{\lambda \EE[g(z) - g(z')\mid z_1\ldots,z_{i}]} \mid z_1, \ldots, z_{i-1}\right]\\
	&\leq \EE\left[ e^{\lambda (g(z) - g(z'))} \mid z_1, \ldots, z_{i-1}\right]\\
	&= \EE\left[ \frac{1}{2}e^{\lambda (g(z) - g(z'))} + \frac{1}{2}e^{\lambda (g(z') - g(z))}\mid z_1, \ldots, z_{i-1}\right]\\
	&= \EE\left[ \cosh(\lambda(g(z) - g(z')))\mid z_1, \ldots, z_{i-1}\right]\\
	&= \EE\left[ \cosh(\lambda|g(z) - g(z')|)\mid z_1, \ldots, z_{i-1}\right]\\
	&\leq \EE\left[ \cosh(\lambda(\omega_i(z_i, z_i')))\mid z_1, \ldots, z_{i-1}\right]\\
	&= \EE\left[ \frac{1}{2}e^{\lambda (\omega_i(z_i, z_i'))} + \frac{1}{2}e^{-\lambda (\omega_i(z_i, z_i'))}\right]\\
	&= \EE\left[e^{\lambda\varepsilon_i \omega_i(z_i, z_i')}\right] = e^{\psi_{\varepsilon_i\omega_i}(\lambda)},
	\end{align*}
	where the first inequality follows by Jensen's inequality and the tower rule.
	Appealing to $\eqref{eqn:induction}$, and using induction, we therefore conclude
	$$\EE[e^{\lambda Y_m}]\leq  e^{\psi_{\varepsilon_m\omega_m}(\lambda)} \EE[e^{\lambda Y_{m-1}}]\leq \ldots\leq e^{\lambda \EE[g(z)]+\sum_{i=1}^m \psi_{\varepsilon_i\omega_i}(\lambda)}.$$
	Thus \eqref{eqn:moment_gen} is proved. The estimate \eqref{eqn:prox_bound} then follows by the standard Cram\'{e}r-Chernoff bounding method. Namely, assume $\omega_i=\omega_j$ for all indices $i$ and $j$. Then for every $t\geq 0$, Chernoff's inequality \cite[Page 21]{boucheron2013concentration} together with \eqref{eqn:moment_gen} implies
	\begin{equation}
	\mathbb{P}(g(z)-\EE[g(z)]\geq t)\leq e^{-(m \psi_{\varepsilon\omega})^{\star}(t)}.
	\end{equation}
	Noting the equality $(m \psi_{\varepsilon\omega})^{\star}(t)=m\psi_{\varepsilon\omega}^{\star}(\frac{t}{m})$ completes the proof.
\end{proof}

The final ingredient is to perform a covering argument over the points $y$. To this end, we will need the following theorem that guarantees Lipschitz continuity of the proximal map in the metric $d_{\Phi}(x,y):= \|\nabla \Phi(x) - \nabla \Phi(y)\|_\ast$.

\begin{lem}[Lipschitz continuity]\label{lem:lip_prox_bregman_final}
	Consider a closed and lower-bounded function $g\colon\R^d\to\R\cup\{\infty\}$ that is $\rho$-weakly convex and compatible with some Legendre function $\Phi\colon\R^d\to\{\infty\}$.  Assume in addition that $\Phi$ is $\alpha$-strongly convex on $\dom g\cap \inter(\dom \Phi)$. Then for any $\bar \rho>\rho$ and $x,y\in \inter(\dom \Phi)$, we have
	\begin{align*}
	\|\prox^{\Phi}_{g/\bar \rho}(x) - \prox^{\Phi}_{g/\bar\rho}(y)\| \leq \frac{\bar\rho}{\alpha(\bar\rho - \rho)}\|\nabla \Phi(x) - \nabla \Phi(y)\|_\ast.
	\end{align*}
\end{lem}
\begin{proof}
	For any points $x, y$, set $\hat x=\prox^{\Phi}_{g/\bar\rho}(x)$ and $\hat y=\prox^{\Phi}_{g/\bar\rho}(y)$.
	Taking into account that $g+\bar\rho\Phi$ is $(\bar\rho-\rho)$ strongly convex relative to $\Phi$, we deduce
	\begin{align*}
	(\bar\rho - \rho) D_{\Phi}( \hat y,\hat x) &\leq \left(g(\hat y) + \bar\rho D_{\Phi}(\hat y, x)\right) -\left(g(\hat x)+\bar\rho D_{\Phi}(\hat x,x)\right) \\
	(\bar\rho - \rho) D_{\Phi}( \hat x,\hat y) 
	&\leq \left(g(\hat x)  + \bar\rho D_{\Phi}(\hat x, y)\right) -\left(g(\hat y)+\bar\rho D_{\Phi}(\hat y,y) \right).
	\end{align*}
	Adding these estimates together, we obtain 
	\begin{align*}
	\alpha(\bar\rho - \rho)\|\hat x - \hat y\|^2
	&\leq(\bar\rho - \rho)(D_{\Phi}( \hat y,\hat x) + D_{\Phi}( \hat x,\hat y)   \\
	&\leq\bar\rho \left(D_{\Phi}(\hat y, x) -  D_{\Phi}(\hat x,x)  +  D_{\Phi}(\hat x, y) -   D_{\Phi}(\hat y,y)\right)\\
	&= \bar\rho \dotp{\nabla \Phi(x) - \nabla \Phi(y), \hat x - \hat y} \\
	&\leq \bar\rho \|\nabla \Phi(x) - \nabla \Phi(y)\|^\ast \|\hat x - \hat y\|.
	\end{align*}
	The result follows by dividing both sides by $\|\hat x - \hat y\|$.
\end{proof}

We now have all the ingredients to prove the main result of this section. To this end, for any set $C\subseteq\inter(\dom \Phi)$, we will let $\cN(C,\Phi, \delta)$ denote the covering number of $C$ in the metric $(x,y)\mapsto \|\nabla \Phi(x) - \nabla \Phi(y)\|_\ast$.

\begin{thm}[Concentration of the stationarity measure]\label{thm:prox_close_p}
	Let $C\subseteq\inter(\dom \Phi)$ be any set. Then with probability $$1-\cN(C,\Phi, \delta) \exp\left( -m\cdot\psi^\star_{\varepsilon L}\left(\frac{s}{\sqrt{m}}\right)\right),$$  we have
	$$\sup_{y\in C}~\|\prox_{\varphi_S/\bar\rho}^\Phi(y)-\prox_{\varphi/\bar\rho}^\Phi(y)\|\leq \sqrt{\tfrac{4(\sigma+s)^2}{\alpha(\bar\rho-\rho)^2 m}}+\tfrac{2\bar\rho\delta}{\alpha(\bar\rho-\rho)}.$$
	If  $\Phi$ is in addition twice differentiable on $\inter(\dom \Phi)$, then with the same probability, the estimate holds: 
	$$\sup_{y\in C}~\|\nabla(\varphi_S)_{1/2\rho}^\Phi(y)-\nabla \varphi_{1/2\rho}^\Phi(y)\|^*_x\leq  \sqrt{\tfrac{16(\sigma+s)^2}{\alpha m}}+\tfrac{8\rho\delta}{\alpha}.$$
\end{thm}
\begin{proof}
	Following the notation of Theorem~\ref{thm:stab_ERM}, set $$\alg(y, S):=\prox^{\Phi}_{\varphi_S/\bar\rho}(y)\qquad\textrm{and} \qquad\alg^*(y):=\prox^{\Phi}_{\varphi/\bar\rho}(y).$$
	Define the function
	$$
	g(y,S) :=  \|\alg(y, S) - \alg^\ast(y)\|.
	$$
	We will first apply Theorem~\ref{thm:McDiarmid} to the function $g(y,\cdot)\colon \Omega^m\to\R$, with $y$ fixed. To verify the bounded difference property, we compute 
	\begin{align}
	|g(y,S) - g(y,S^{i})| &= \left|\|\alg(y,S) - \alg^\ast(y)\| -\|\alg(y,S^{i}) - \alg^\ast(y)\|\right|,\notag \\
	&\leq \|\alg(y,S) - \alg(y,S^{i})\|, \label{eqn:triangel}\\
	&\leq \frac{L(z_i) + L(z_i')}{\sqrt{\alpha}(\bar\rho-\rho)m},\label{eqn:lkapply_lemma}
	\end{align}
	where \eqref{eqn:triangel} uses the reverse triangle inequality, while \eqref{eqn:lkapply_lemma} follows from the estimate \eqref{eqn:stab_guarant} and strong convexity of $\Phi$. 
	Setting $\omega(z_i,z_i')=\frac{L(z_i) + L(z_i')}{\sqrt{\alpha}(\bar\rho-\rho)m}$, we deduce 
	$$\psi_{\varepsilon\omega}(\lambda)=\psi_{\varepsilon L}\left(\frac{2\lambda}{\sqrt{\alpha}(\bar\rho-\rho)m}\right)\qquad \textrm{and}\qquad \psi_{\varepsilon\omega}^\star(t)=\psi_{\varepsilon L}^\star\left(\sqrt{\alpha}(\bar\rho-\rho)m t/2\right).$$
	Note moreover from \eqref{eqn:comparsiontoopt} the bound $\EE g(y,S)=\EE \|\alg(y, S) - \alg^\ast(y)\|\leq \sqrt{\frac{4\sigma^2}{\alpha(\bar\rho-\rho)^2 m}}$. Thus, applying  Theorem~\ref{thm:McDiarmid}, we conclude
	$$\mathbb{P}\left(g(y,S)\geq \sqrt{\tfrac{4\sigma^2}{\alpha(\bar\rho-\rho)^2 m}}+t\right)\leq \exp\left( -m\psi^{\star}_{\varepsilon L}\left(\sqrt{\alpha}(\bar\rho-\rho)t/2\right)\right).
	$$
	Next we show using Lemma~\ref{lem:lip_prox_bregman_final} that $g(\cdot, S)$ is $\tfrac{2\bar\rho}{\alpha(\bar\rho-\rho)}$-Lipschitz with respect to the metric $d_{\Phi}(x,y)$. Indeed, observe
	\begin{align*}
	|g(y, S) - g(y', S)| &\leq \Big|\|\alg(y,S)-\alg^*(y)\|-\|\alg(y',S)-\alg^*(y')\|\Big|\\
	&\leq \|\alg(y, S) - \alg(y', S)\|+\|\alg^*(y) - \alg^*(y')\| \leq \tfrac{2\bar\rho}{\alpha(\bar\rho-\rho)}\bregmetric(y, y'),
	\end{align*}
	where we used the triangle inequality and Lipschitz continuity of the proximal operator (Lemma~\ref{lem:lip_prox_bregman_final}). 
	Let  $\{y_i\}$ be a $\delta$-net of $C$ in the metric $d_{\Phi}$. Thus for every $y$ in a $\delta$-ball of $y_i$, we have $g(y,S)\leq g(y_i,S)+\frac{2\bar\rho\delta}{\alpha(\bar\rho-\rho)}$. Taking a union bound over the net, we therefore deduce
	$$\mathbb{P}\left(g(y,S)\leq \sqrt{\tfrac{4\sigma^2}{\alpha(\bar\rho-\rho)^2 m}}+\tfrac{2\bar\rho\delta}{\alpha(\bar\rho-\rho)}+t\right)\geq 1-\cN(C,\Phi, \delta) \exp\left( -m\psi^\star_{\varepsilon L}\left(\sqrt{\alpha}(\bar\rho-\rho)t/2\right)\right).$$
	Setting $t=\frac{2s}{\sqrt{\alpha m}(\bar\rho-\rho)}$ completes the proof.
\end{proof}

We end the section by showing how Theorem~\ref{thm:prox_close_p} directly implies analogous bounds on a localized Hausdorff distance between the subdifferential graphs, $\gph \partial \varphi$ and $\gph \partial \varphi_S$.

\begin{thm}[Concentration of subdifferential graphs]
	Let $C\subseteq\R^d$ be any set and let $r>0$ and $\bar{\rho}>\rho$ be arbitrary.
	Then with probability $$1-\cN\left(C+r\bar\rho {\bf B},\Phi, \delta\right) \exp\left( -m\cdot\psi^\star_{\varepsilon L}\left(\frac{s}{\sqrt{m}}\right)\right),$$ the estimates hold
	\begin{align}
	(C\times r{\bf B})\cap\gph\partial \varphi_S&\subset	\gph\partial \varphi+\left(\sqrt{\tfrac{4(\sigma+s)^2}{\alpha(\bar\rho-\rho)^2 m}}+\tfrac{2\bar\rho\delta}{\alpha(\bar\rho-\rho)}\right)\left({\bf B}\times  \bar\rho{\bf B}\right) \label{eqn:subdiff1},\\
	(C\times r{\bf B})\cap\gph\partial \varphi&\subset	\gph\partial \varphi_S+\left(\sqrt{\tfrac{4(\sigma+s)^2}{\alpha(\bar\rho-\rho)^2 m}}+\tfrac{2\bar\rho\delta}{\alpha(\bar\rho-\rho)}\right)\left({\bf B}\times \bar\rho{\bf B}\right) \label{eqn:subdiff2}.
	\end{align}

\end{thm}

\begin{proof}
	Fix a pair of points $x,y\in \R^d$ and observe the equivalence
	$$y=\prox_{\varphi_S/\bar \rho}(x)\quad\Longleftrightarrow\quad  \bar{\rho}^{-1}(x-y)\in \partial \varphi_S(y).$$
	Let $y\in C$ and $v\in r{\bf B}\cap \partial \varphi_S(y)$ be arbitrary.
	Define the point $x:=y+\bar\rho v$. Clearly then we may write $y=\prox_{\varphi_S/\bar \rho}(x)$ and the inclusion $x\in C+r\bar\rho {\bf B}$ holds. Appealing to Theorem~\ref{thm:prox_close_p}, we therefore deduce that with probability $$1-\cN(C+r\bar\rho {\bf B},\Phi, \delta) \exp\left( -m\cdot\psi^\star_{\varepsilon L}\left(\frac{s}{\sqrt{m}}\right)\right),$$ simultaneously for all $y\in C$ and $v\in  r{\bf B}\cap \partial\varphi_S(y)$ and $\delta>0$, we have
	$$\|y-\prox_{\varphi/\bar\rho}(x)\|\leq \sqrt{\tfrac{4(\sigma+s)^2}{\alpha(\bar\rho-\rho)^2 m}}+\tfrac{2\bar\rho\delta}{\alpha(\bar\rho-\rho)}.$$
	Set $y':=\prox_{\varphi/\bar\rho}(x)$ and 
	$v':=\bar\rho^{-1}(x-y')\in \partial \varphi(z),$
	and observe
	$$\frac{1}{\bar\rho}\|v-v'\|=\|y-y'\|\leq \sqrt{\tfrac{4(\sigma+s)^2}{\alpha(\bar\rho-\rho)^2 m}}+\tfrac{2\bar\rho\delta}{\alpha(\bar\rho-\rho)}.$$
	Thus we showed $$(y,v)\in (y',v')+\left(\sqrt{\tfrac{4(\sigma+s)^2}{\alpha(\bar\rho-\rho)^2 m}}+\tfrac{2\bar\rho\delta}{\alpha(\bar\rho-\rho)}\right)({\bf B}\times {\bar\rho}{\bf B}).$$
	The inclusion~\eqref{eqn:subdiff1} follows immediately, while \eqref{eqn:subdiff2} follows by a symmetric argument.
\end{proof}

We next instantiate  Theorem~\ref{thm:prox_close_p} in the $\ell_2$ and $\ell_1$ settings. We will require the use of the sub-Gaussian norm of any random variable $X$, which is defined to be 
$\|X\|_{sg}:=\inf\{t>0: \EE\exp(X^2/t^2)\leq 2\}$, along with the sub-exponential norm $\|X\|_{se}:=\inf\{t>0: \EE\exp(|X|/t)\leq 2\}$. Recall also that given any norm $\|\cdot\|$, the $\delta$-covering number of the unit ball $\{x:\|x\|\leq 1\}$ in the norm $\|\cdot\|$ is at most 
$\left(1+\frac{2}{\delta}\right)^d$. 

\subsection*{The Euclidean $\ell _2$ set-up.}
Let us instantiate Theorem~\ref{thm:prox_close_p} in the Euclidean setting $\Phi=\frac{1}{2}\|\cdot\|^2_2$ with $C$ whose diameter in $\ell_2$-norm is bounded by $B$. 

\paragraph{Sub-Gaussian Lipschitz constant:} Suppose that $L-\EE L$ is a sub-Gaussian random variable with parameter $\nu=\|L-\EE[L]\|_{sg}$. Using the triangle inequality, we therefore deduce
$$\|\varepsilon L\|_{sg}=\|L\|_{sg}\leq \|L-\EE[L]\|_{sg}+ \|\EE[L]\|_{sg}\leq \nu+\frac{\sigma}{\sqrt{\ln 2}}.$$
Appealing to \cite[Equation 2.16]{vershynin2018high}, we conclude  
$\psi_{\varepsilon L}(\lambda)\leq \frac{c}{2}\left(\nu+\sigma\right)^2\lambda^2$ for all $\lambda\in \R$, where $c$ is a numerical constant. Taking conjugates yields the relation $\psi_{\varepsilon L}^{\star}(t)\geq \frac{t^2}{c\left(\nu+\sigma\right)^2}$. Appealing to Theorem~\ref{thm:prox_close_p}, while setting $s=\sqrt{c(\nu+\sigma)^2\ln\left(\frac{\cN(C,\Phi, \delta)}{\gamma}\right)}$ and $\delta =\frac{1}{\rho}\sqrt{\frac{d}{m}}$, we deduce that with probability
$1-\gamma$, the estimate holds: 
$$\sup_{y\in C}~\|\nabla(\varphi_S)_{1/2\rho}(y)-\nabla \varphi_{1/2\rho}(y)\|_2\lesssim \sqrt{\frac{\sigma^2+d}{ m}+\frac{(\nu+\sigma)^2 d}{m}\ln\left(\tfrac{R}{\gamma}\right)},$$
where we set $R:=1+2\rho B\sqrt{\frac{m}{d}}$.

\paragraph{Globally bounded Lipschitz constant:}
As the next example, suppose that there exists a constant $\lipsymb$ satisfying $L(z)\leq \lipsymb$ for a.e. $z\in \Omega$. Then clearly we have $\sigma\leq L$ and $\nu:=\|L-\EE[L]\|_{sg}\lesssim \lipsymb$. Consequently, we deduce that with probability
$1-\gamma$, the estimate holds: 
$$\sup_{y\in C}~\|\nabla(\varphi_S)_{1/2\rho}(y)-\nabla \varphi_{1/2\rho}(y)\|_2\lesssim \sqrt{\frac{\lipsymb^2+d}{ m}+\frac{\lipsymb^2 d}{m}\ln\left(\tfrac{R}{\gamma}\right)}.$$
where we set $R:=1+2\rho B\sqrt{\frac{m}{d}}$.

\paragraph{Sub-exponential Lipschitz constant:}
As the final example, we suppose that $L-\EE[L]$ is sub-exponential and set $\nu=\|L-\EE[L]\|_{se}$. A completely analogous argument as in the sub-Gaussian case implies $\|\varepsilon L\|_{se}\leq \nu+\frac{\sigma}{\ln(2)}$. Appealing to \cite[Proposition 2.7.1]{vershynin2018high}, we deduce 
$\psi_{\varepsilon L}(\lambda)\leq c ^2(\nu+\sigma)^2\lambda^2$ for all $|\lambda|\leq \frac{1}{c(\nu+\sigma)}$. To simplify notation set $\eta:=c(\nu+\sigma)$. Taking conjugates, we therefore deduce 
$$\psi_{\varepsilon L}^{\star}(t)\geq \begin{cases} 
\frac{t^2}{4\eta^2} & \textrm{if } |t|\leq 2\eta  \\
\frac{|t|}{\eta}-1 & \textrm{otherwise} \\
\end{cases}.$$
Consequently, we deduce the usual Bernstein-type bound
$\psi_{\varepsilon L}^{\star}(t)\geq \min \{\frac{t^2}{4\eta^2},\frac{t}{2\eta}\}$. Setting 
$s=2\eta\cdot\max\left\{\sqrt{\ln\left(\frac{\cN(C,\Phi, \delta)}{\gamma}\right)},\frac{1}{\sqrt{m}}\ln\left(\frac{\cN(C,\Phi, \delta)}{\gamma}\right)\right\}$ and $\delta=\frac{1}{\rho}\sqrt{\frac{d}{m}}$ in Theorem~\ref{thm:prox_close_p}, we deduce that with probability $1-\gamma$, we have
$$\sup_{y\in C}~\|\nabla(\varphi_S)_{1/2\rho}(y)-\nabla \varphi_{1/2\rho}(y)\|_2\lesssim  \sqrt{\frac{\sigma^2+d}{m}+(\nu+\sigma)^2\max\left\{\frac{d}{m}\log\left(\tfrac{R}{\gamma}\right),\frac{d^2}{m^2}\log^2\left(\tfrac{R}{\gamma}\right)\right\}},$$
where we set $R:=1+2\rho B\sqrt{\frac{m}{d}}$.

\subsection*{The $\ell_1$ set-up.}
Suppose now we use the $\ell_1$ set-up as in Example~\ref{ex:non_euclid}. Namely set
$\Phi(x)=\frac{1}{2(p-1)}\|x\|^2_p$ with $p=1+\frac{1}{\ln d}$. Noting that $\Phi$ is  $1/e^2$-strongly convex with respect to the $\ell_1$-norm, we may set $\alpha=1/e^2$. Note that $\Phi$ is not twice differentiable; indeed $\nabla \Phi$ is only H\"{o}lder continuous. Consequently, we will focus on the uniform estimates on the quantity $\|\prox_{\varphi_S/2\rho}^\Phi(y)-\prox_{\varphi/2\rho}^\Phi(y)\|_1$, instead of the gradient of the Bregman envelopes. 

Let $C\subset\R^d$ now be an arbitrary set whose diameter in the $\ell_1$ norm is upper bounded by $B$. To simplify the arithmetic, we'll suppose throughout $B\geq 1$.
We will need to estimate the covering number $\cN(C,\Phi, \delta)$. To this end, using \cite[eqn. 5.5]{complexity}, we have the estimate\footnote{The result \cite[eqn. 5.5]{complexity}, as stated, applies to a slight modification of $\Phi$; it, however, immediately implies \eqref{eqn:needed_holder} by using positively homogeneity of $\nabla\Phi$.}
\begin{equation}\label{eqn:needed_holder}
\|\nabla \Phi(y)-\nabla \Phi(x)\|_{\infty}\leq 4\cdot (2B)^{\frac{1}{\ln (d)}} \ln (d)\cdot\|y-x\|^{\frac{1}{\ln (d)}}_1.
\end{equation}
Therefore $\cN(C,\Phi, \delta)$ is upper bounded by the $\frac{1}{2B}\left(\frac{\delta}{4\ln d}\right)^{\ln d}$-covering number of $C$ in the norm $\|\cdot\|_1$. Hence a quick computation shows
$\cN(C,\Phi, \delta)\leq (1+\frac{32B^2\ln (d)}{\delta})^{d\ln d}$.

\paragraph{Sub-Gaussian Lipschitz constant:}
Suppose that $L-\EE L$ is a sub-Gaussian random variable with parameter $\nu=\|L-\EE[L]\|_{sg}$. Consequently, as  noted above, we have 
$\psi_{\varepsilon L}^{\star}(t)\geq \frac{t^2}{c\left(\nu+\sigma\right)^2}$ for all $t\in \R$, where $c$ is a numerical constant. As in the Euclidean setting, appealing to Theorem~\ref{thm:prox_close_p} with $s=\sqrt{c(\nu+\sigma)^2\ln\left(\frac{\cN(C,\Phi, \delta)}{\gamma}\right)}$ and $\delta =\frac{1}{\rho}\sqrt{\frac{d}{m}}$, we deduce that with probability
$1-\gamma$, the estimate holds: 
$$\sup_{y\in C}~\|\prox_{\varphi_S/2\rho}^\Phi(y)-\prox_{\varphi/2\rho}^\Phi(y)\|_1 \lesssim \frac{1}{\rho}\cdot\sqrt{\frac{\sigma^2+d}{  m}+\frac{(\nu+\sigma)^2 d\ln(d)}{ m}\ln\left(\tfrac{R}{\gamma}\right)},$$
where we set $R:=1+32\rho B^2\sqrt{\frac{m\ln^2(d)}{d}}$.

\paragraph{Globally bounded Lipschitz constant:}
As the next example, suppose that there exists a constant $\lipsymb$ satisfying $L(z)\leq \lipsymb$ for a.e. $z\in \Omega$. Then the same reasoning as in the Euclidean setting yields the guarantee: with probability
$1-\gamma$, the estimate holds: 
$$\sup_{y\in C}~\|\prox_{\varphi_S/2\rho}^\Phi(y)-\prox_{\varphi/2\rho}^\Phi(y)\|_1 \lesssim \frac{1}{\rho}\cdot \sqrt{\frac{\lipsymb^2+d}{ m}+\frac{\lipsymb^2 d\ln(d)}{m}\ln\left(\tfrac{R}{\gamma}\right)}.$$
where we set $R:=1+32\rho B^2\sqrt{\frac{m\ln^2(d)}{d}}$.

\paragraph{Sub-exponential Lipschitz constant:}
As the final example, we suppose that $L-\EE[L]$ is sub-exponential and set $\nu=\|L-\EE[L]\|_{se}$. A completely analogous argument as in the Euclidean case yields the guarantee: with probability
$1-\gamma$, the estimate holds: 
$$\sup_{y\in C}~\|\prox_{\varphi_S/2\rho}^\Phi(y)-\prox_{\varphi/2\rho}^\Phi(y)\|_1 \lesssim \tfrac{1}{\rho}\cdot  \sqrt{\tfrac{\sigma^2+d}{m}+(\nu+\sigma)^2\max\left\{\tfrac{d\log d}{m}\log\left(\tfrac{R}{\gamma}\right),\left(\tfrac{d\log (d)}{m}\right)^2\log^2\left(\tfrac{R}{\gamma}\right)\right\}},$$
where we set $R:=1+32\rho B^2\sqrt{\frac{m\ln^2(d)}{d}}$.

\section{Dimension Independent Rates.}\label{sec:dim_indep}
In this section, we introduce a technique for obtaining bounds on the graphical distance between subdifferentials from estimates on the closeness of function values. The main result we use is a quantitative version of the Attouch convergence theorem from variational analysis~\cite{attouch1,attouch2}. A variant of this theorem was recently used by the authors to analyze the landscape of the phase retrieval problem~\cite[Theorem 6.1]{davis2017nonsmooth}. For the sake of completeness, we present a short argument, which incidentally simplifies the original exposition in~\cite{davis2017nonsmooth}. 

The approach of this section has benefits and drawbacks. The main benefit is that because we obtain subdifferential distance bounds via closeness of values, whenever function values uniformly converge at a dimension independent rate, so do the subdifferentials. This type of result is in stark contrast to the results in Section~\ref{sec:dimension_dep}, which scale with the dimension. The main drawback is that for losses that uniformly converge at a rate of $\delta$, we can only deduce subdifferential bounds on the order of ${O}(\sqrt{\delta})$, yielding what appear to be suboptimal rates. Nevertheless, the very existence of dimension independent bounds is notable. We will illustrate the use of the techniques on two examples: learning with generalized linear models (Section~\ref{sec:multiclass}) and robust nonlinear regression (Section~\ref{sec:robust_nonlinear}).

\begin{thm}[Subdifferential graphs]\label{thm:compar}
	Consider two closed and lower-bounded functions $g,h\colon \R^d\to\R\cup\{\infty\}$, having identical domain $\mathcal{D}$, and which are $\rho$-weakly convex and compatible with a Legendre function $\Phi\colon\R^d\to\R\cup\{\infty\}$. Suppose moreover that $\Phi$ is $\alpha$-strongly convex relative to some norm $\|\cdot\|$ on $ \mathcal{D}\cap \inter(\dom \Phi)$ and for some real $l,u\in \R$, the inequalities hold: 
	\begin{equation}\label{eqn:uniform_closeness}
	l\leq h(x)-g(x)\leq u, \qquad \forall x\in \mathcal{D}\cap \inter(\dom \Phi).
	\end{equation}
	Then for any $\bar\rho>\rho$ and any point $x\in \inter(\dom\Phi)$, the   estimate holds:
	\begin{align}
	\bregsym\left(\prox_{g/\bar\rho}^\Phi(x),\prox_{h/\bar\rho}^\Phi(x)\right)\leq \frac{u-l}{\bar\rho-\rho},\label{eqn:breg_close_root}
	\end{align}
	and therefore, provided that $\Phi$ is twice differentiable on $\inter(\dom\Phi)$,  we have
	\begin{equation}
	\|\nabla g_{1/\bar\rho}^\Phi(x)-\nabla h_{1/\bar\rho}^\Phi(x)\|_x^{*}\leq \bar\rho \sqrt{\frac{u-l}{\alpha(\bar\rho-\rho)}}\label{eqn:breg_close_root_triangle}.
	\end{equation}
	Moreover, in the Euclidean setting $\Phi=\frac{1}{2}\|\cdot\|^2_2$, we obtain the estimate:
	\begin{equation}\label{eqn:subdiff_close}
	\dist_{1/\bar \rho}\left(\gph \partial g,\gph \partial h\right)\leq \sqrt{\frac{u-l}{\bar\rho-\rho}},
	\end{equation}
	where the Hausdorff distance $\dist_{1/\bar \rho}(\cdot,\cdot)$ is induced by the norm $(x,v)\mapsto \max\left\{\|x\|_2,\frac{1}{\bar \rho}\|v\|_2\right\}$.
\end{thm}
\begin{proof} Fix a  point $x\in \inter(\dom\Phi)$ and define 
	$x_g:=\prox_{g/\bar\rho}^{\Phi}(x)$ and $x_h:=\prox_{h/\bar\rho}^{\Phi}(x)$. We successively deduce
	\begin{align}
	g(x_g)+\bar\rho D_{\Phi}(x_g,x)&\leq \left(g(x_h)+\bar\rho D_{\Phi}(x_h,x)\right)-(\bar\rho-\rho)D_{\Phi}(x_h,x_g)\label{eqn:sconv1}\\
	&\leq h(x_h)+\bar\rho D_{\Phi}(x_h,x)-(\bar\rho-\rho)D_{\Phi}(x_h,x_g)-l\label{eqn:useupperlower1}\\
	&\leq h(x_g)+\bar\rho D_{\Phi}(x_g,x)-(\bar\rho-\rho)\bregsym(x_h,x_g)-l\label{eqn:sconv2}\\
	&\leq g(x_g)+\bar\rho D_{\Phi}(x_g,x)-(\bar\rho-\rho)\bregsym(x_h,x_g)+(u-l),\label{eqn:useupperlower2}
	\end{align}
	where \eqref{eqn:sconv1} and \eqref{eqn:sconv2} follow from strong convexity of $g(\cdot)+\bar\rho D_{\Phi}(\cdot,x)$ and $h(\cdot)+\bar\rho D_{\Phi}(\cdot,x)$, respectively, while \eqref{eqn:useupperlower1} and \eqref{eqn:useupperlower2} follow from the assumption \eqref{eqn:uniform_closeness}. Rearranging immediately yields \eqref{eqn:breg_close_root}. The inequality \eqref{eqn:breg_close_root_triangle} follows directly from \eqref{eqn:breg_close_root} and Theorem~\ref{thm:env_diff}.
	
	Consider now the Euclidean setting $\Phi=\frac{1}{2}\|\cdot\|^2_2$ and fix an arbitrary pair $(x,v)\in \gph \partial g$. A quick computation shows then $x=\prox_{g/\bar\rho}(x+\frac{1}{\bar\rho}v)$. Define now $x':=\prox_{h/\bar\rho}(x+\frac{1}{\bar\rho}v)$ and $v'=\bar\rho(x-x'+\frac{1}{\bar\rho}v)$, and note the inclusion $v'\in \partial h(x')$. Appealing to \eqref{eqn:breg_close_root}, we therefore deduce $\|x'-x\|_2\leq \sqrt{\frac{u-l}{\bar\rho-\rho}}$ and $\|v'-v\|_2=\bar{\rho}\|x-x'\|_2\leq \bar \rho  \sqrt{\frac{u-l}{\bar\rho-\rho}}$. We have thus shown $\dist_{1/\bar \rho}((x,v),\gph\partial h)\leq \sqrt{\frac{u-l}{\bar\rho-\rho}}.$ A symmetric argument reversing the roles of $f$ and $g$ completes the proof of \eqref{eqn:subdiff_close}.
\end{proof}

Note that simple examples of uniformly close functions, such as $h(x) = \delta \sin (\delta^{-1/2} x)$ and $g(x) = 0$, show that the guarantee of Theorem~\ref{thm:compar} is tight. 

In a typical application of Theorem~\ref{thm:compar} to subgradient estimation, one might set $h$ to be the population risk and $g$ to be the empirical risk or vice versa. The attractive feature of this approach is that it completely decouples probabilistic arguments (for proving functional convergence) from variational analytic arguments (for proving graphical convergence of subdifferentials). The following two sections illustrate the use of Theorem~\ref{thm:compar} on two examples: learning with generalized linear models (Section~\ref{sec:multiclass}) and robust nonlinear regression (Section~\ref{sec:robust_nonlinear}).

	\subsection{Illustration I: Dimension Independent Rates for Generalized Linear Models}\label{sec:multiclass}

In this section, we develop \emph{dimension-independent} convergence guarantees for a wide class of generalized linear models. We consider a loss functions $f : \RR^d \times \Omega \rightarrow \Omega$ over a bounded set $\cX$, where $f(x, z)$ has the parametric form
$$
f(x, z) = \ell(\dotp{x, \phi_1(z)}, \ldots, \dotp{x, \phi_K(z)},z),
$$
Here $\ell \colon \RR^K\times\Omega \rightarrow \RR$ is a loss function and $\phi_1, \ldots, \phi_K$ are feature maps. For simplicity, the results of this section are only derived in the $\ell_2$ norm, though some of the results hold in greater generality. We make the following assumptions:
\begin{enumerate}
\item[(C1)] (\textbf{Norms}) We equip $\RR^d$ and $\RR^K$ with the $\ell_2$ norms $\|\cdot\|_2$. 
\item[(C2)] (\textbf{Region of Convergence}) We assume that $\cX \subseteq \RR^d$ is a closed set containing a point $x_0 \in \cX$. We assume that $\sup_{ x\in \cX} \| x - x_0\|_2 \leq B$ for a constant $B > 0$.
\item[(C3)] (\textbf{Feature Mapping}) The feature maps $\phi_k \colon \Omega \rightarrow \RR^d$ are measurable for $k = 1, \ldots, K$.
\item[(C4)] (\textbf{Loss Function and Regularizer}) $\ell \colon \RR^K \times \Omega \rightarrow \RR$ is a measurable function. We assume that for each $z \in \Omega$, the function $\ell(\cdot, z)$ is $L(z)$-Lipschitz over the set $$\{ (\dotp{x, \phi_1(z)}, \ldots, \dotp{x, \phi_K(z)}) \mid x \in \cX\}$$ for a  measurable map $ L \colon\Omega \rightarrow \RR_+$. The function $r \colon \RR^d \rightarrow \RR\cup \{\infty\}$ is lower semi-continuous.
\end{enumerate}
Then we have the following theorem, whose proof is Presented in Appendix~\ref{sec:radamacher} 
	\begin{thm}[Dimension Independent Functional Concentration]\label{thm:glm_main}
	Let $z_1, \ldots, z_n, z, z'$ be an i.i.d sample from $P$ and define the random variable
	$$
	Y = \left[ |f(x_0, z) - f(x_0, z')|+ BL(z) \sqrt{\sum_{k=1}^K\|\phi_k(z)\|^2} + BL(z') \sqrt{\sum_{k=1}^K\|\phi_k(z')\|^2}\right].
	$$ 
	Then under assumptions (C1)-(C4), with probability 
	$$1 - 2\exp\left(-m\psi^\star_{\varepsilon Y}(t)\right),$$
	 we have following bound:
	\begin{align*}
	\sup_{x \in \cX} \left| \frac{1}{m}\sum_{i=1}^m f(x, z_i)  - \EE\left[ f(x, z)\right]\right| \leq 2\sqrt{\frac{2B^2K \displaystyle\max_{k=1, \ldots, K}\EE_z\left[ L(z)^2\|\phi_k(z)\|_2^2\right]}{ m}} + t.
	\end{align*}
	\end{thm}
	
Thus far we have not assumed any weak convexity of the function $f(\cdot, z)$. In order to prove concentration of the subdifferential graphs, we now explicitly make this assumption: 
\begin{enumerate}
\item [(C5)] (\textbf{Weak Convexity With High Probability}) There exists a constant $\rho > 0$ and a probability $p_m\in [0, 1]$ such that with probability $1-p_m$ over the sample $S = \{z_1, \ldots, z_m\}$, the functions
\begin{align*}
	\varphi(x) := \EE\left[ f(x, z) \right] + r(x) + \iota_{\cX}(x) && \text{ and } && \varphi_S(x) := f_S(x) + r(x) + \iota_{\cX}(x).
\end{align*}
are $\rho$-weakly convex relative to $\Phi(x) = \frac{1}{2}\|x\|_2^2$.
\end{enumerate}
Given these assumptions, we may deduce subdifferential convergence with Theorem~\ref{thm:compar}---the main result of this section.
\begin{cor}[Dimension Independent Rates for GLMs]\label{cor:dim_indep_rateGLM}
	Assume the setting of Theorem~\ref{thm:glm_main} and Assumptions (C1)-(C5).  
	Let $z_1, \ldots, z_m, z, z'$ be an i.i.d. sample from $P$ and define the random variable
	$$
	Y =  2L(z)B \sqrt{\sum_{k=1}^K\|\phi_k(z)\|_2^2}.
	$$ 
	 Then with probability 
	$$1 - 2\exp\left(-m\psi^\star_{\varepsilon Y}(t)\right)-p_m,$$
	we have the following bound:
	\begin{align*}
	\sup_{x\in \cX}\,\|\nabla \varphi_{1/2\rho}-\nabla (\varphi_S)_{1/2\rho}(x)\|&\leq \sqrt{\frac{\bar\rho}{\bar \rho-\rho}}\cdot \sqrt{\sqrt{\frac{32B^2K \displaystyle\max_{k=1, \ldots, K}\EE_z\left[ L(z)^2\|\phi_k(z)\|_2^2\right]}{ m}} + 2t},	\\	
	\dist_{1/\bar\rho}(\gph \partial \varphi, \gph\partial \varphi_S)&\leq \frac{1}{\sqrt{\bar\rho-\rho}}\cdot\sqrt{\sqrt{\frac{32B^2K \displaystyle\max_{k=1, \ldots, K}\EE_z\left[ L(z)^2\|\phi_k(z)\|_2^2\right]}{ m}} + 2t},
	\end{align*}
	where the Hausdorff distance $\dist_{1/\bar \rho}(\cdot,\cdot)$ is induced by the norm $(x,v)\mapsto \max\left\{\|x\|_2,\frac{1}{\bar \rho}\|v\|_2\right\}$.

	\end{cor}
	\begin{proof}
	We will apply Theorem~\ref{thm:glm_main} after shift. Namely set $$\bar l(s,z)=l(s,z)-l(\langle x_0,\phi_1(z)\rangle,\ldots,\langle x_0,\phi_K(z)\rangle )$$
	and define the loss
	$\bar f(x,z)=\bar l(\dotp{x, \phi_1(z)}, \ldots, \dotp{x, \phi_K(z)},z)$. Define now the functions $\bar \varphi(x) = \varphi(x)- \EE\left[f(x_0, z)\right]$ and $\bar \varphi_S = \varphi_S(x) - \frac{1}{m}\sum_{z \in S} f(x_0, z)$.
	Applying Theorem~\ref{thm:glm_main} to $\bar f(x,z)$, we deduce that with probability $1 - 2\exp\left(-m\psi^\star_{\varepsilon Y}(t)\right),$ we have
	\begin{align*}
	\sup_{x \in \cX} \left|\bar \varphi_S(x) - \bar \varphi(x) \right| \leq 2\sqrt{\frac{2B^2K \displaystyle\max_{k=1, \ldots, K}\EE\left[ L(z)^2\|\phi_k(z)\|_2^2\right]}{ m}} + t.
	\end{align*}
	Thus, due to assumption (C5), we may apply Theorem~\ref{thm:compar} to the functions $\bar \varphi(x)$ and $\bar \varphi_S$, noticing that $\partial \bar \varphi(x) = \partial \varphi(x)$ and $\partial \bar \varphi_S(x) = \partial \varphi_S(x)$, as desired.
	\end{proof}
 
 If the random variable $2L(z)B \sqrt{\sum_{k=1}^K\|\phi_k(z)\|_2^2}$ is subgaussian, we immediately obtain a dimension independent $m^{-1/4}$ rate of convergence. This is in stark contrast to all other results obtained in this paper.

\subsection{Illustration II: Landscape of Robust Nonlinear Regression.}\label{sec:robust_nonlinear}

In this section, we investigate a  robust nonlinear regression problem in $\RR^d$, using the techniques we have developed. Setting the stage, consider a function $\sigma\colon\R^d\times \Omega\to\R$ that is differentiable in its first component and let $\bar x$ be the ground truth. Our observation model is 
$$b(z,\delta,\xi)=\sigma\left(\langle \bar x,z\rangle, z\right)+\delta\xi,$$
where $z,\delta$ and $\xi$ are random variables. One should think of $z$ as the population data, $\delta$ as encoding presence or absence of an outlier, and $\xi$ as the size of the outlying measurement. Seeking to recover $\bar x$, we consider the formulation 
$$\min_{x\in \mathcal X}~ f(x,z):=\EE_{z,\delta,\xi}[|\sigma\left(\langle x,z\rangle, z\right)-b(z,\delta,\xi)|]$$
where the set $\mathcal{X}$ will soon be determined. We make the following assumptions on the data.
\smallskip

\begin{description}
	\item[{\rm (D1)}] (\textbf{Sufficient Support}) There exist constants $c, C > 0$, such 
	for all $x \in \RR^d$, we have
	$$
	C^2\|x\|^2_2 \geq \EE\left[|\dotp{x, z}|^2\right],  \qquad \EE\left[|\dotp{x, z}|\right] \geq c\|x\|_2 \qquad\text{ and } \qquad  P(\dotp{x, z} \neq 0) = 1.
	$$
	\item[{\rm(D2)}] (\textbf{Corruption Frequency}) $\delta$ is a $\{0, 1\}$-valued random variable. We define $$\pfail  :=P(\delta = 1),$$ which is independent from $z$ and $\xi$.
	\item[{\rm(D3)}] (\textbf{Finite Moment}) $\xi$ is a random variable with finite first moment. 
	\item[{\rm(D4)}] (\textbf{Lipschitz, Smooth, and Monotonic Link}) There exist constants $a > 1$ and $c_\sigma,C_\sigma>0$  satisfying $c_\sigma \leq \sigma'(u, z) \leq C_\sigma$ for all $u \in  \{ \dotp{x, z} \mid \|x\|_2 \leq a\|\bar x\|_2\}$ and $z\in \Omega$. In addition, for every $z \in \Omega$ the function $\sigma'(\cdot, z)$ is $L$-Lipschitz continuous.
	\item [{\rm(D5)}] (\textbf{Concentration}) Let $p_m\in [0,1]$ and $\tau_m>0$ be sequences  satisfying
	$$
	\mathbb{P}_S\left( \left\| \frac{1}{m} \sum_{z \in S}^m zz^T\right\|_{\text{op}}  \leq \tau_m \right)\geq  1- p_m.
	$$
	where $S = \{z_1, \ldots, z_m\}$ is an i.i.d. sample from $P$.
\end{description}
\smallskip

The noise model considered above allows for adversarial corruption, meaning that $\xi$ may take the form $\xi = \sigma(\dotp{x_0, z}, z) - \sigma(\dotp{\bar x, z}, z)$ for an arbitrary point $x_0$. This allows us to ``plant" a completely different signal in the measurements. The rest of the assumptions serve to make $\bar x$ identifiable from the measurements $\sigma(\dotp{\bar x, z}, z)$, as we will soon show. 

The goal of this section is to prove the following theorem, which shows that the empirical risk is well-behaved. In particular, the empirical risk is weakly convex and its stationary points cluster around $\bar x$.
\begin{thm}[Stationary Points of the Empirical Risk]\label{thm:example_main}
	Define $\cX = a\|\bar x\|_2 \bf B$. For any sample $S \subseteq \Omega$ of size $m$, set
	$$
	\varphi(x) := f(x) + \iota_{\cX}(x) \qquad\textrm{and}\qquad \varphi_S(x) := f_S(x) + \iota_{\cX}(x).
	$$
	Then $\varphi$ is $2LC^2$-weakly convex and with probability $1-p_m$ the function $\varphi_S$ is $2L\tau_m$-weakly convex. 
	Suppose now  $\pfail< \frac{c_{\sigma}c}{2C_{\sigma}C}$ and set 
	$$
	\rho = \max\{2LC^2, 2L\tau_m\}\qquad \textrm{and}\qquad D = c_\sigma c - 2\pfail C_\sigma C.
	$$
	Then whenever $t>0$ and $m$ satisfy
	\begin{align*}
	t \leq \frac{1}{256\rho}D^2&& \text{and}&& m \geq \frac{2^{21}\rho^2 C_\sigma^2a^2 \|\bar x\|^2_2\EE\left[ \|z\|^2_2\right]}{D^4},
	\end{align*}
	we have, with probability 
	$$
	1 - 2\exp\left(-m\psi^\star_{\varepsilon   \|z\|_2}\left(\frac{t}{2a\|\bar x\|_2 C_\sigma}\right)\right) - p_m,
	$$ that 
	any pair $(x, v) \in \gph \varphi_S$ satisfies at least one of the following: 
	\begin{enumerate}
		\item (\textbf{Near global optimality})
		$$
		\| x -\bar x\|_2 \leq \frac{16}{    D}\cdot \left(\sqrt{\frac{8a^2\|\bar x\|^2_2 C_\sigma^2\EE\left[ \|z\|^2_2\right]}{ m}} + t\right).
		$$

		\item (\textbf{Large Subgradient})
		$$
		\|v\|_2 \geq \frac{1}{2}D.
		$$
	\end{enumerate}
\end{thm}

Let us briefly examine Assumptions (D1)-(D5) and the conclusion of the theorem in the case of a Gaussian population $z \sim N(0, I_{d \times d})$. Assumption (D1) holds true with $C = 1$ and $c = \sqrt{2/\pi}$. Assumption (D2)-(D4) are independent of the distribution of $z$. Assumption (D5) holds true with 
\begin{align*}
\tau_m  = 4 + \frac{d}{m}+4\sqrt{\frac{d}{m}} && \text{and} && p_m = 2\exp(-m/2),
\end{align*}
by Corollary~\cite[Corollary 5.35]{vershynin2010introduction}. Thus, assumption (D1)-(D5) are satisfied. Now we examine the various quantities included in the theorem. 

The expected squared norm of a gaussian is  $\EE\left[\|z\|^2_2\right] = d$. One can also show, using standard probabilistic techniques, that the moment generating function satisfies the bound 
$$
\psi_{\varepsilon \|z\|_2} (t) \leq  \frac{d\kappa t^2}{2},
$$
for a numerical constant $\kappa > 0$. Thus, we find that $ \psi^\star_{\varepsilon   \|z\|_2}(t) \geq \frac{ t^2}{2d\kappa}$. Therefore, by equating 
$$
\frac{\delta}{2} = \exp\left(\frac{-mt^2}{2d\kappa(2a\|\bar x\|_2 C_\sigma)^2}\right)
$$
and solving for $t$, we find that with probability $1-\delta - p_m$, every pair $(x, v)\in \gph \varphi_S$ satisfies 
\begin{align*}
\|x - \bar x\|_2 = O\left(  \sqrt{\frac{ a^2 \|\bar x\|^2_2 C_\sigma^2d}{m}\log\left(\frac{1}{\delta}\right)}\right) && \text{ or } && \|v\|_2 \geq \frac{1}{2}\left(c_\sigma \sqrt{\frac{2}{\pi}} - 2\pfail C_\sigma \right).
\end{align*}
Interestingly, although Theorem~\ref{thm:compar} in general provides rates of convergence that scale as $m^{-1/4}$ as shown in Corollary~\ref{cor:dim_indep_rateGLM}, we obtain standard statistical rates of convergence for $\|\bar x - x\|_2$.  This would not be possible with a direct application of Theorem~\ref{thm:prox_close_p}, as we would obtain rates that scale as $\sqrt{d^2/m}$. Finally, we note that for this bound to be useful, we must have corruption frequency $\pfail$ strictly less than $\frac{c_{\sigma}}{C_{\sigma}}\sqrt{\frac{1}{2\pi}}$.

We now present the proof of Theorem~\ref{thm:example_main}.
\begin{proof}[Proof of Theorem~\ref{thm:example_main}]
	
	Although $\varphi$ is nonsmooth and nonconvex, it is fairly well-behaved. We first show that $\varphi$ and $\varphi_S$ are both weakly convex. 
	
	\begin{claim}[Weak Convexity]\label{lem:robust_weak}
		The functions $f$ and $\varphi$ are $2LC^2$-weakly convex. Moreover with probability $1-p_m$ the functions $f_S$ and $\varphi_S$ are $2L\tau_m$-weakly convex.
	\end{claim}
	\begin{proof}[Proof of Claim~\ref{lem:robust_weak}] 
		For any fixed $x, z, \xi, \delta$, by the mean value theorem, there exists $\eta$ in the interval $[\dotp{x, z}, \dotp{y, z}]$, so that for all $y \in \RR^d$, we have
		\begin{equation}\label{eqn:weak_conv_arg}
		\begin{aligned}
		&|\sigma(\dotp{y, z}, z) - \sigma(\dotp{\bar x, z}, z) + \xi \cdot \delta|\\
		&= |\sigma(\dotp{x, z}, z) + \sigma'(\eta, z)\dotp{y -  x, z} - \sigma(\dotp{\bar x, z}, z) + \xi \cdot \delta|\\
		&\geq |\sigma(\dotp{x, z}, z) + \sigma'(\dotp{x, z}, z)\dotp{y -  x, z} - \sigma(\dotp{\bar x, z}, z) + \xi \cdot \delta| \\
		&~~- |\sigma'(\eta, z) - \sigma'(\dotp{x, z}, z)||\dotp{y -  x, z}|\\
		&\geq |\sigma(\dotp{x, z}, z) + \sigma'(\dotp{x, z}, z)\dotp{y -  x, z} - \sigma(\dotp{\bar x, z}, z) + \xi \cdot \delta| - L|\dotp{y -  x, z}|^2.
		\end{aligned}
		\end{equation}
		Therefore, taking expectations we deduce
		\begin{align*}
		f(y ) &\geq \EE\left[|\sigma(\dotp{x, z}, z) + \sigma'(\dotp{x, z}, z)\dotp{y -  x, z} - \sigma(\dotp{\bar x, z}, z) + \xi \cdot \delta|\right] - L\EE\left[|\dotp{y -  x, z}|^2\right]\\
		&\geq \EE\left[|\sigma(\dotp{x, z}, z) + \sigma'(\dotp{x, z}, z)\dotp{y -  x, z} - \sigma(\dotp{\bar x, z}, z) + \xi \cdot \delta|\right] - LC^2\|y - x\|^2_2.
		\end{align*}
		Notice that the right rand-side is a $2LC^2$-weakly convex function in $y$. We have thus deduced that for every $x$, there is a $2LC^2$-weakly convex function that globally lower bounds $f(\cdot)$ while agreeing with it at $x$. Therefore $f$ is $2LC^2$-weakly convex, as claimed.

		Next, using~\eqref{eqn:weak_conv_arg} yields the inequality:
		\begin{align*}
		f_S(y) &\geq \frac{1}{m} \sum_{z \in S} |\sigma(\dotp{x, z}, z) + \sigma'(\dotp{x, z}, z)\dotp{y -  x, z} - \sigma(\dotp{\bar x, z}, z) + \xi \cdot \delta|- \frac{L}{m} \sum_{z \in S}|\dotp{y -  x, z}|^2.
		\end{align*}
		Finally, notice with probability $\tau_m>0$ we get the upper bound:
		\begin{align*}
		\frac{L}{m} \sum_{z \in S}|\dotp{y -  x, z}|^2 \leq L  \left\|\frac{1}{m} \sum_{z \in S} zz^T\right\|_{\text{op}}\cdot \|y - x\|^2_2 \leq  \tau_m \cdot L\|y - x\|^2_2.
		\end{align*}
		By the same reasoning as for the population objective, we deduce that $f_S$ is $2L\tau_m$-weakly convex with probability $p_m$, as claimed. 
	\end{proof}
	
	Having established weak convexity, we now lower bound the subgradients of $f$ and show that for all $x \neq \bar x$, the negative subgradients of $f$ always point toward $\bar x$. In particular, the point $\bar x$ is the unique stationary point of $f$. 
	\begin{claim}[Stationarity conditions for $f$]\label{lem:grad_robust_bound}
		For every $x \neq \bar x$ and $v \in \partial f(x)$, we have
		$$
		(c_\sigma c - 2\pfail C_\sigma C)\cdot \|x - \bar x\|_2 \leq \dotp{ v, x - \bar x},
		$$
		and consequently
		$$
		\|v\|_2 \geq c_\sigma c - 2\pfail C_\sigma C.
		$$
	\end{claim}
	\begin{proof}[Proof of Claim~\ref{lem:grad_robust_bound}]
		For every $x \in \RR^d$, define a measurable mapping $\zeta_0 (x, \cdot) : \Omega \rightarrow \RR^d$ by
		\begin{align*}
		\zeta_0(x, z ) := \sigma'(\dotp{x, z}, z)\sign\left(\sigma(\dotp{x, z}, z) - \sigma(\dotp{\bar x, z}, z)\right) \cdot z.
		\end{align*}
		Now, observe that  
		\begin{align*}
		\EE\left[ \dotp{\zeta_0(x, z), x - \bar x}\right] &= \EE\left[ \sigma'(\dotp{x, z},z)\sign(\sigma(\dotp{x, z},z) - \sigma(\dotp{\bar x, z},z))\dotp{z, x- \bar x} )\right]\\
		&= \EE\left[ \sigma'(\dotp{x, z},z)|\dotp{z, x- \bar x}|\right]\\
		&\geq c_\sigma\EE\left[ |\dotp{z, x- \bar x}|\right]\\
		&\geq c_\sigma c \cdot \|x - \bar x\|_2,
		\end{align*}
		where the second equality follows from monotonicity of $\sigma(\cdot, z)$.

		As each term $| \sigma(\dotp{x, z}, z) - \sigma(\dotp{\bar x, z}, z) + \delta \cdot \xi|$ is subdifferentially regular (each term is Lipschitz and weakly convex by Claim~\ref{lem:robust_weak}), it follows that 
		\begin{align*}
		\partial f(x) = \left\{ \EE\left[ \zeta(x, (z, \xi, \delta))\right] \mid \zeta(x, (z, \xi, \delta)) \in \partial_x (|\sigma(\dotp{\cdot, z}, z) - \sigma(\dotp{\bar x, z}, z) + \delta \cdot \xi |)(x) \text{ a.e.}\right\},
		\end{align*}
		where the set definition ranges over all possible  $ \zeta(x, \cdot) : \Omega \rightarrow \RR^d$ that are also measurable~\cite[Theorem 2.7.2]{clarke}. 
		Next, we claim that for any such measurable mapping, we have
		\begin{equation}\label{eqn:restriction}
		\EE\left[ \zeta(x, (z, \xi, \delta)) - \zeta_0(x,z)) \mid \delta = 0\right] = 0.
		\end{equation}
		To see this, observe that the function $\EE|\sigma(\dotp{\cdot, z}, z) - \sigma(\dotp{\bar x, z}, z)|$ is differentiable at any $x \neq \bar x$, since $\mathbb{P}(\dotp{y, z} = 0) = 1$. It follows that the subdifferential of this function at any $x \neq \bar x$ consists only of the expectation of the measurable selection $\zeta_0$. The claimed equality~\eqref{eqn:restriction} follows. 
		
		Thus, by  linearity of expectation  and the inclusion $\zeta(x, (z, \xi, \delta)), \zeta_0(x, z) \in \sigma'(\dotp{x, z}, z) [-1, 1] z$, we have
		\begin{align*}
		\dotp{\EE\left[\zeta(x, (z, \xi, \delta)) - \zeta_0(x,z )\right], x - \bar x}
		&= (1-\pfail)\dotp{\EE\left[\zeta(x, (z, \xi, \delta)) - \zeta_0(x, z)) \mid \delta = 0\right], x - \bar x}\\
		& \hspace{20pt}+ \pfail \dotp{\EE\left[\zeta(x, (z, \xi, \delta)) - \zeta_0(x, z) \mid \delta = 1\right], x - \bar x}\\
		&= \pfail \dotp{\EE\left[\zeta(x, (z, \xi, \delta)) - \zeta_0(x, z) \mid \delta = 1\right], x - \bar x}\\
		&\geq -\pfail \EE\left[ 2\sigma'(\dotp{x, z}, z)|\dotp{z, x - \bar x}|\right]\\
		&\geq -2\pfail C_\sigma C \cdot \|x - \bar x\|_2.
		\end{align*}
		Therefore, we arrive at the bound: 
		\begin{align*}
		\dotp{\EE\left[\zeta(x, (z, \xi, \delta))\right], x - \bar x} &= \dotp{\EE\left[\zeta_0(x, z)\right], x - \bar x} + \dotp{\EE\left[\zeta(x, (z, \xi, \delta)) - \zeta_0(x, z)\right], x - \bar x} \\
		&\geq c_\sigma c  \|x - \bar x\|_2 -  2\pfail C_\sigma C \cdot \|x - \bar x\|_2\\
		&= (c_\sigma c - 2\pfail C_\sigma C)\cdot \|x - \bar x\|_2.
		\end{align*}
		As every element of $\partial f(x)$ is of the form $\EE\left[\zeta(x, (z, \xi, \delta))\right]$, the proof is complete. 
	\end{proof}
	
	While the only stationary point of $f$ is $\bar x$, it is as yet unclear where the (random) stationary points of $f_S$ lie, since we can only guarantee that that the functional deviation $|f - f_S|$ is small on bounded sets. Thus, we first show that constraining $f$ to a ball containing $\bar x$ does not create any extraneous stationary points at the boundary of the ball.
	\begin{claim}[Constrained Stationary Conditions of $f$]\label{claim:constrained}
		Let $a > 1$ be a fixed constant. Let $x \in a\|\bar x\|_2\bf B$ be such that $x \neq \bar x$. Then for every $v \in \partial f(x) +  N_{a\|\bar x\|_2\bf B}(x)$, we have 
		$$
		(c_\sigma c - 2\pfail C_\sigma C)\cdot \|x - \bar x\|_2 \leq \dotp{v, x - \bar x}
		$$
		and consequently
		$$
		\|v\|_2 \geq c_\sigma c - 2\pfail C_\sigma C.
		$$
	\end{claim}
	\begin{proof}[Proof of Claim~\ref{claim:constrained}]
		By Claim~\ref{lem:grad_robust_bound}, we must only consider the case when $\|x\|_2 = a\|\bar x\|_2$ since otherwise $N_{a\|\bar x\|_2\bf B}(x) = \{0\}$ and $v \in \partial f(x)$. In this case, we have 
		$$
		v = v_f + \lambda x
		$$
		where $v_f \in \partial f(x)$ and $\lambda \geq 0$. Therefore, we find that 
		\begin{align*}
		\dotp{v, x - \bar x} &= \dotp{v_F, x - \bar x} + \dotp{\lambda x, x - \bar x}\\
		&\geq \dotp{v_F, x - \bar x} + \lambda \|x\|^2_2 - \lambda \dotp{x, \bar x}\\
		&\geq \dotp{v_F, x - \bar x} + \lambda a^2 \|\bar x\|^2_2 - \lambda a\|\bar x\|^2_2\geq \dotp{v_F, x - \bar x}.
		\end{align*}
		Thus, applying Claim~\ref{lem:grad_robust_bound} completes the proof.
	\end{proof}
	
	Finally, we may now examine the stationary points of $ f_S$ constrained to a ball. We show that every nearly stationary point of $f_S + \delta_\cX$ must be within a small ball around $\bar x$.
	To that end, we define 
	$$
	\eta = \sqrt{\frac{32a^2 \|\bar x\|^2_2 C_\sigma^2\EE\left[ \|z\|^2_2\right]}{ m}} + 2t.
	$$
	Notice that for all $z \in \Omega$, the function $\sigma(\cdot, z)$ is $L(z) = C_\sigma$ Lipschitz. In addition, every point in $\cX = a\|\bar x\|_2 \bf B$ is bounded in norm by $a\|\bar x\|_2$. Therefore, by Corollary~\ref{cor:dim_indep_rateGLM} with $x_0=0$, we have that with probability 
	$$1 - 2\exp\left(-m\psi^\star_{\varepsilon   \|z\|_2}\left(\frac{t}{2a\|\bar x\|_2 C_\sigma}\right)\right)-p_m,
	$$
	the bound holds:
	$$
	\dist_{1/\bar \rho}\left(\gph \partial \varphi,\gph \partial \varphi_S\right)\leq \sqrt{\frac{\eta}{\bar\rho-\rho}},
	$$
	where we set $\rho := \max\{2LC^2, 2L\tau_m\}$ and $\bar \rho>\rho$ is arbitrary.
	
	In particular, for any $\gamma > 0$, setting 
	$
	\bar \rho = \frac{\gamma^2}{\eta} + \rho, 
	$
	we deduce that for any pair $(x,v)\in \gph\partial \varphi_S$ there exists  a point $\hat x \in \cX$ and a subgradient $\hat v \in \partial \varphi(\hat x)$ satisfying  
	\begin{align*}
	\|x - \hat x\|_2 \leq \eta/\gamma && \text{and} && \|v - \hat v \|_2 \leq \bar \rho \cdot \eta/\gamma = \gamma + \rho\eta/\gamma. 
	\end{align*}
	Let us choose $\gamma > 0$ so that 
	$
	\gamma + \rho\eta /\gamma \leq \frac{1}{2}D,
	$
	which may be accomplished by finding a root of the polynomial 
	$
	\gamma^2 - \frac{1}{2}D\gamma + \rho\eta = 0.
	$
	Thus by the quadratic formula, we have
	$
	\gamma = \frac{ \frac{1}{2}D + \sqrt{ \frac{1}{4}D^2 - 4\rho\eta}}{2}. 
	$
	Notice that by our assumptions on $t$ and $m$, we have
	$
	4\rho\eta \leq \frac{1}{8}D^2,
	$
	and therefore we deduce
	$\frac{D}{4}\leq\gamma.$
	Thus by Claim~\ref{claim:constrained}, if $\hat x \neq \bar x$, there exists $\hat v \in \partial \varphi (\hat x) = \partial f(\hat x) + N_{\cX}(\hat x) $ such that 
	$$
	\|v\|_2 \geq \|\hat v\|_2 - \|v - \hat v\|_2 \geq (c_\sigma c - 2\pfail C_\sigma C)  - \frac{1}{2}(c_\sigma c - 2\pfail C_\sigma C) = \frac{1}{2}(c_\sigma c - 2\pfail C_\sigma C).
	$$
	Otherwise, $\hat x = \bar x$ and 
	$
	\| x - \bar x\|_2 \leq \eta/\gamma\leq \frac{4\eta}{D},
	$
	as desired. 
\end{proof}

\bibliographystyle{plain}
\bibliography{bibliography}

\appendix\section{Rademacher Complexity and Functional Bounds.}\label{sec:radamacher}

	In this section, we use the well-known technique for bounding the suprema of empirical processes, based on Rademacher complexities (see e.g.,~\cite{bartlett2002rademacher,bartlett2005local}). We will use these bounds to obtain concentration inequalities for multi-class generalized linear models. Although such arguments have become standard in the literature, we present a proof that explicitly uses Theorem~\ref{thm:McDiarmid} in order to obtain a slightly more general result for unbounded classes. None of the results or techniques here are new; rather, the purpose of this section is to keep the paper self-contained. We begin with the following standard definition.
	\begin{definition}{\rm 
			The {\em Rademacher complexity} of a set $A\subset\R^m$ is the quantity
			$$\mathcal{R}(A)=\frac{1}{m}\EE_{\varepsilon}\left[\sup_{a\in A}\,\langle \varepsilon,a\rangle\right].$$
			where the coordinates of $\varepsilon\in\R^m$ are i.i.d.\ Rademacher random variables.
		}
	\end{definition}
	Given a collection of functions $\mathcal{G}$ from $\Omega$ to $\R$ and a set $S=\{z_1,\ldots, z_m\}\subset\Omega$, we define 
	$$\mathcal{G}\circ S:=\{(g(z_1),\ldots,g(z_m)): g\in \mathcal{G}\}.$$

	The following theorem shows that the Rademacher complexity directly controls uniform convergence of the sample average approximation.
	
	\begin{thm}\label{thm:bound_sup}
		Consider a countable class  $\cG$ of measurable functions from $\Omega$ to $\RR$ and let $S=\{z_1, \ldots, z_m\}$ be an i.i.d. sample from $P$. Define the random variable 
		$$
		Y = \sup_{g \in \cG} |g(z) - g(z')|,
		$$ 
		for independent copies $z,z'\sim P$ and let $\varepsilon$ be a Rademacher random variable.
		Then for all $t > 0$, with probability $$1 - 2\exp\left(-m\psi^\star_{\varepsilon Y}(t)\right),$$ we have the following bound:
		\begin{align*}
		\sup_{g \in \cG} \left| \EE_z\left[g(z)\right]-\frac{1}{m} \sum_{i=1}^m g(z_i) \right| \leq  2\EE_{S}\mathcal{R}(\mathcal{G}\circ S)+t.
		\end{align*} 
	\end{thm}
	
	\begin{proof}
		Define the two random variables
		$X^+=\sup_{g \in \cG} \left\{ \EE_z\left[g(z)\right]-\frac{1}{m} \sum_{i=1}^m g(z_i)\right\}$ and $X^-=\sup_{g \in \cG} \left\{\frac{1}{m} \sum_{i=1}^m g(z_i)- \EE_z\left[g(z)\right]\right\}$. We first bound the expectations of $X^+$ and $X^-$.
		Appealing to \cite[Lemma 26.2]{shalev2014understanding} we deduce $\EE[X^+]\leq 2\EE_{S}\mathcal{R}(\mathcal{G}\circ S)$. Replacing $\mathcal{G}$ with $-\mathcal{G}$ and using \cite[Lemma 26.2]{shalev2014understanding}, we also learn $\EE[X^-]\leq 2\EE_{S}\mathcal{R}(\mathcal{-G}\circ S)=2\EE_{S}\mathcal{R}(\mathcal{G}\circ S)$. 
		Next, a quick computation shows 
		\begin{align*}
		&|X^+(z_1, \ldots, z_m) - X^+(z_1, \ldots, z_{i-1}, z_i', z_{i+1}, \ldots, z_m)| \leq \frac{1}{m}\sup_{g \in \cG} \left|g(z_i) - g(z_i')\right|=\frac{1}{m}Y,
		\end{align*}
		as well as the analogous inequality for $X^-$. Thus using Theorem~\ref{thm:McDiarmid}, we conclude that with probability $1- 2\exp(-m\psi^\star_{m^{-1}\varepsilon Y}(t/m))$, we have $\max\{X^+,X^-\}\leq 2\EE_S\mathcal{R}(\mathcal{G}\circ S)+t$. Noting the equality 
		$\psi^\star_{m^{-1}\varepsilon Y}(t/m)=\psi^\star_{\varepsilon Y}(t)$ completes the proof.
	\end{proof}

	The following theorem provides an upper bound on the Rademacher complexity of linear classes; see the original article \cite{kakade2009complexity} or the monograph \cite[Lemma 26.10]{shalev2014understanding}.
	
	\begin{lem}[Rademacher complexity of linear classes]\label{lem:rad_comp_lin_class}{\hfill \\ }
		Consider the set $A=\{(\langle w,z_1\rangle,\ldots,\langle w,z_m\rangle): \|w\|_2\leq 1\}$, where $z_1\ldots, z_m$ are arbitrary points. Then the estimate holds:
		$$\mathcal{R}(A)\leq \sqrt{\frac{\sum_{i=1}^m \|z_i\|^2_2}{m}}.$$\end{lem}
	
	The class of loss functions $\cG$ considered below will be formed from compositions of functions with linear classes. A useful result for unraveling such compositions is the following vector-valued contraction inequality, recently proved by Maurer~\cite{vector_contraction}.
	
	\begin{thm}[Contraction Inequality~{\cite[Theorem 3]{vector_contraction}}]\label{thm:vector_contraction}
		Let $\cX$ denote a countable set. For $i = 1, \ldots, m$, let $F_i \colon \cS \rightarrow \RR$ and $G_i \colon \cS \rightarrow \R^K$ be functions satisfying
		$$
		F_i(s)  - F_i(u)  \leq \|G_i(s) - G_i(u) \|_2 \qquad \text{for all $s, u \in \cS$}.
		$$
		Define the two sets
		\begin{equation*}
		F\circ \cS=\{(F_1(s),\ldots,F_m(s)): s\in \cS \}\qquad \textrm{ and }\qquad
		G\circ \cS =\{\,\left(G^k_i(s)\right)_{i,k}: s\in \mathcal{S}\},
		\end{equation*}
		where $G^k_i(s)$ denotes the $k$'th coordinate of $G_i(s)$.
		Then the estimate holds:
		\begin{align*}
		\mathcal{R}(F\circ \cS) \leq \sqrt{2}K\cdot  \mathcal{R}(G\circ \cS).
		\end{align*}	  
	\end{thm}
	
	We are now ready to prove Theorem~\ref{thm:glm_main}. 
	
	\begin{proof}[Proof of Theorem~\ref{thm:glm_main}]
		We will apply Theorem~\ref{thm:bound_sup}, to the function class
		$$
		\cG = \{ z \mapsto f(x, z) \mid x \in \cX\}.
		$$
		We note that, due to the separability of $\RR^d$ and the continuity of the integrands, any supremum over all $x \in \cX$ may be replaced by a supremum over a countable dense subset of $\cX$, without affecting its value. We ignore this technicality throughout the proof.	
		
		As the first step in applying Theorem~\ref{thm:bound_sup}, we compute 
		\begin{align*}
		&\sup_{x \in \cX} |f(x, z) - f(x, z') | \\
		&\leq |f(x_0, z) - f(x_0, z')|+ L(z) \sup_{x \in \cX} \sqrt{\sum_{k=1}^K \dotp{x - x_0, \phi_k(z)}^2} +  L(z') \sup_{x \in \cX} \sqrt{\sum_{k=1}^K \dotp{x - x_0, \phi_k(z')}^2}\\
		&\leq |f(x_0, z) - f(x_0, z')|+ BL(z) \sqrt{\sum_{k=1}^K\|\phi_k(z)\|^2_2} + BL(z') \sqrt{\sum_{k=1}^K\|\phi_k(z')\|^2_2},
		\end{align*}
		where the last inequality uses the bound $\|x - x_0\|_2 \leq B$ twice. Notice the right-hand-side is precisely the random variable $Y$.
		
		Next we upper bound the expected Rademacher complexity $\EE_S\cR(\cG\circ S)$ by using Theorem~\ref{thm:vector_contraction}. To this end, fix a sample set $S=\{z_1,\ldots,z_m\}$ and define
		$$\cS=\{(\langle x,\phi_k(z_i)\rangle)_{i,k}: x\in \cX\}.$$	
		For every index $s \in \cS$ and $i \in \{1, \ldots, m\}$, set $s_i:=(s_{i1},\ldots, s_{iK})$ and define the functions
		$F_i(s) = \ell(s_{i}, z_i)$ and $G_i(s) =L(z_i)s_{i}$. 
		We successively compute
		\begin{equation}\label{eqn:main_ineq}
		\begin{aligned}
		\mathcal{R}(\mathcal{F}\circ S)&=\frac{1}{m}\sup_{x\in \cX} \sum_{i=1}^m \sigma_i f(x,z_i)\\
		&=\frac{1}{m} \sup_{x\in \cX} \sum_{i=1}^m \sigma_i l\left((\langle x,\phi_1(z_i),\ldots,\langle x,\phi_K(z_i)\rangle),z_i\right) \\
		&=\frac{1}{m} \sup_{s\in \cS} \sum_{i=1}^m \sigma_i F_i(s)=\mathcal{R}(F\circ \cS)\leq \sqrt{2}K\cdot  \mathcal{R}(G\circ \cS),
		\end{aligned}
		\end{equation}
		where the last inequality follows from Theorem~\ref{thm:vector_contraction}. 
		
		Next, unraveling notation, observe $G\circ \cS=\{(\langle x,L(z_i)\phi_k(z_i)\rangle)_{i,k}: x\in \cX\}$. 
		Moreover, shifting and shrinking $\mathcal{X}$, it follows directly from the definition of Rademacher complexity that $\mathcal{R}(\cG\circ \cS)=B\cdot\mathcal{R}(A')$ where we set $A'=\{\left(\langle x,L(z_i)\phi_k(z_i)\rangle\right)_{i,k}: \|x\|_2\leq 1\}$. Thus applying Lemma~\ref{lem:rad_comp_lin_class}, we deduce 
		$\mathcal{R}(\cG\circ \cS)\leq \frac{\sqrt{\sum_{i,k}B^2L(z_i)^2\|\phi_k(z_i)\|^2_2}}{mK}.$
		Combining this estimate with \eqref{eqn:main_ineq} and
		taking expectations yields
		$$\displaystyle\EE_{S}\mathcal{R}(\cG\circ S)\leq \frac{\sqrt{2\sum_{i,k}B^2\EE_{z_i}[L(z_i)^2\|\phi_k(z_i)\|^2_2]}}{m}=\sqrt{\frac{2B^2K\displaystyle\max_{k=1,\ldots,K}\EE_{z}[L(z)^2\|\phi_k(z)\|^2_2]}{m}}.$$
		Appealing to Theorem~\ref{thm:bound_sup} completes the proof.
	\end{proof}

\end{document}